\documentclass[11pt]{article}
\usepackage{amsmath}
\usepackage{amssymb}
\usepackage{amsthm}
\usepackage{amsfonts}
\usepackage{color}

\newcounter{AbcT}

\newtheorem {Theorem}    {Theorem}[section]

\newtheorem {Lemma}      [Theorem]    {Lemma}
\newtheorem {Corollary}  [Theorem]    {Corollary}
\newtheorem {Proposition}[Theorem]    {Proposition}

\theoremstyle{definition}   
\newtheorem {Remark}[Theorem]           {Remark}  

\newcommand{\ignore}[1]{}

\newcommand{\sgn}{{\mbox{{\rm sgn}}}}

\def\eps{\epsilon}
\def\E{{\bf{E}}}
\def\P{{\bf{P}}}
\def\Var{{\bf{Var}}}

\def\R{{\mathbb{R}}}

\def\CalN{{\cal{N}}}
\def\CalM{{\cal{M}}}

\def\F{{\cal{F}}}

\def\v0{{\bf 0}}

\def\0{\hat{0}}
\def\1{\hat{1}}

\def\path{{\tt path}}

\def\phi{\varphi}

\def\be{\begin{equation}}
\def\ee{\end{equation}}

\definecolor{Red}{rgb}{1,0,0}
\definecolor{Blue}{rgb}{0,0,1}
\definecolor{Olive}{rgb}{0.41,0.55,0.13}
\definecolor{Green}{rgb}{0,1,0}
\definecolor{MGreen}{rgb}{0,0.8,0}
\definecolor{DGreen}{rgb}{0,0.55,0}
\definecolor{Yellow}{rgb}{1,1,0}
\definecolor{Cyan}{rgb}{0,1,1}
\definecolor{Magenta}{rgb}{1,0,1}
\definecolor{Orange}{rgb}{1,.5,0}
\definecolor{Violet}{rgb}{.5,0,.5}
\definecolor{Purple}{rgb}{.75,0,.25}
\definecolor{Brown}{rgb}{.75,.5,.25}
\definecolor{Grey}{rgb}{.5,.5,.5}
\definecolor{Black}{rgb}{0,0,0}

\title{A Quantitative Arrow Theorem}

\author{Elchanan Mossel\thanks{Weizmann Institute and U.C. Berkeley. Supported by an Alfred Sloan fellowship
  in Mathematics, by NSF CAREER grant DMS-0548249 (CAREER), by DOD ONR grant (N0014-07-1-05-06), by BSF grant 2004105 and by ISF grant
  1300/08}}

\begin{document}
\maketitle

\begin{abstract}
Arrow's Impossibility Theorem states that any constitution which satisfies Independence of Irrelevant Alternatives (IIA) and Unanimity and is not a Dictator has to be non-transitive. In this paper we study quantitative versions of Arrow theorem. Consider $n$ voters who vote independently at random, each following the uniform distribution over the $6$ rankings of $3$ alternatives.
Arrow's theorem implies that any constitution which satisfies IIA and Unanimity and is not a dictator has a probability of at least $6^{-n}$ for a non-transitive outcome. When $n$ is large, $6^{-n}$ is a very small probability, and the question arises if for large number of voters it is possible to avoid paradoxes with probability close to $1$.

Here we give a negative answer to this question by proving that for every $\eps > 0$, there exists a $\delta = \delta(\eps) > 0$,
which depends on $\eps$ only, such that for all $n$, and all constitutions on $3$ alternatives, if the constitution satisfies:
\begin{itemize}
\item
The IIA condition.
\item
For every pair of alternatives $a,b$, the probability that the constitution ranks $a$ above $b$ is at least $\eps$.
\item
For every voter $i$, the probability that the social choice function agrees with a dictatorship on $i$ at most $1-\eps$.
\end{itemize}
Then the probability of a non-transitive outcome is at least $\delta$.

Our results generalize to any number $k \geq 3$ of alternatives and to other distributions over the alternatives. We further derive a quantitative characterization of all social choice functions satisfying the IIA condition whose outcome is transitive with probability at least $1-\delta$.  Our results provide a quantitative statement of Arrow theorem and its generalizations and strengthen results of Kalai and Keller who proved quantitative Arrow theorems for $k=3$ and for {\em balanced} constitutions only, i.e., for constitutions which satisfy for every pair of alternatives $a,b$, that the probability that the constitution ranks $a$ above $b$ is {\em exactly} $1/2$.

The main novel technical ingredient of our proof is the use of inverse-hypercontractivity to show that if the outcome is transitive with high probability then there are no two different voters who are pivotal with for two different pairwise preferences with non-negligible probability.
Another important ingredient of the proof is the application of non-linear invariance to lower bound the probability of a paradox for constitutions where all voters have small probability for being pivotal.
\end{abstract}
\newpage

\section{Introduction}
\subsection{Background on Arrow's Theorem}
Arrow's Impossibility theorem~\cite{Arrow:50,Arrow:63} states that certain desired properties cannot hold simultaneously for constitutions on three or more alternatives. Arrow's results were fundamental in the development of social choice theory in Economics. The most celebrated results in this area are Arrow's Impossibility theorem and the Gibbard-Satterthwaite~\cite{Gibbard:73,Satterthwaite:75} Manipulation theorem. Both results demonstrate the non-existence of ranking and voting schemes with very natural properties.

Arrow's theorem demonstrates that IIA, Transitivity and Non-dictatorship, all of which will be defined below, cannot hold simultaneously.
Quantitative versions of Arrow's theorems prove tradeoff between being "close to transitive"
and being "close to a dictator" assuming the IIA property holds. We proceed with a more formal discussion of Arrow's theorem. 

Consider $A = \{a,b,\ldots,\}$, a set of $k \geq 3$ alternatives. A {\em transitive preference} over $A$ is a ranking of the alternatives from top to bottom where ties are not allowed. Such a ranking corresponds to a {\em permutation} $\sigma$ of the elements $1,\ldots,k$ where $\sigma_i$ is the rank of alternative $i$.

A {\em constitution} is a function $F$ that associates to every $n$-tuple $\sigma = (\sigma(1),\ldots,\sigma(n))$ of transitive preferences (also called a {\em profile}), and every pair of alternatives $a,b$ a preference between $a$ and $b$. Some key properties of constitutions include:
\begin{itemize}
\item
{\em Transitivity}. The constitution $F$ is {\em transitive} if $F(\sigma)$ is transitive for all $\sigma$.
In other words, for all $\sigma$ and for all three alternatives $a,b$ and $c$, if $F(\sigma)$ prefers $a$ to $b$, and prefers $b$ to $c$, it also prefers $a$ to $c$. Thus $F$ is transitive if and only if its image is a subset of the permutations on $k$ elements.
\item
{\em Independence of Irrelevant Alternatives (IIA).} The constitution $F$ satisfies the IIA property if for every pair of alternatives $a$ and $b$, the social ranking of $a$ vs. $b$ (higher or lower) depends only on their relative rankings by all voters.
\item
{\em Unanimity.} The constitution $F$ satisfies {\em Unanimity} if the social outcome ranks $a$ above $b$ whenever all individuals rank $a$ above $b$.
\item
The constitution $F$ is a {\em dictator} on voter $i$, if $F(\sigma) = \sigma(i)$, for all $\sigma$, or $F(\sigma) = -\sigma$, for all $\sigma$, where $-\sigma(i)$ is the ranking $\sigma_k(i) > \sigma_{k-1}(i) \ldots \sigma_2(i) > \sigma_1(i)$ by reversing the ranking $\sigma(i)$.
\end{itemize}

Arrow's theorem states~\cite{Arrow:50,Arrow:63} that:
\begin{Theorem} \label{thm:arrow}
Any constitution on three or more alternatives which satisfies Transitivity, IIA and Unanimity is a dictatorship.
\end{Theorem}
It is possible to give a characterization of all constitutions satisfying IIA and Transitivity. Results of Wilson~\cite{Wilson:72} provide a partial characterization for the case where voters are allowed to rank some alternatives as equal. In order to obtain a quantitative version of Arrow theorem, we give an explicit and complete characterization of all constitutions satisfying IIA and Transitivity in the case where all voters vote using a strict preference order.
Write $\F_k(n)$ for the set of all constitutions on $k$ alternatives and $n$ voters satisfying IIA and Transitivity.
For the characterization it is useful write $A >_F B$ for the statement that for all $\sigma$ it holds that $F(\sigma)$ ranks all alternatives in $A$ above all alternatives in $B$. We will further write $F_{A}$ for the constitution $F$ restricted to the alternatives in $A$. The IIA condition implies that $F_A$ depends only on the individual rankings of the alternatives in the set $A$. The characterization of $\F_k(n)$ we prove is the following.

\begin{Theorem} \label{thm:general_arrow}
The class $\F_k(n)$ consist exactly of all constitutions $F$ satisfying the following:
There exist a partition of the set of alternatives into disjoint sets $A_1,\ldots,A_r$ such that:
\begin{itemize}
\item
\[
A_1 >_F A_2 >_F \ldots >_F A_r,
\]
\item
For all $A_s$ s.t. $|A_s| \geq 3$, there exists a voter $j$ such that $F_{A_s}$ is a dictator on voter $j$.
\item
For all $A_s$ such that $|A_s| = 2$, the constitution $F_{A_s}$ is an arbitrary non-constant function of the preferences on the alternatives in $A_s$.
\end{itemize}
\end{Theorem}
We note that for all $k \geq 3$ all elements of $\F_k(n)$ are not desirable as constitutions. Indeed elements of $F_k(n)$ either have dictators whose vote is followed with respect to some of the alternatives, or they always rank some alternatives on top some other. 
For related discussion see~\cite{Wilson:72}. 

The main goal of the current paper is to provide a quantitative version of Theorem~\ref{thm:general_arrow} assuming voters vote
independently and uniformly at random. Note that Theorem~\ref{thm:general_arrow} 
above implies that if $F \not\in \F_k(n)$ then $P(F) \geq (k!)^{-n}$.
However if $n$ is large and the probability of a non-transitive outcome is indeed as small as $(k!)^{-n}$, one may argue that a non-transitive outcome is so unlikely that in practice Arrow's theorem does not hold.

The goal of the current paper is to establish lower bounds on the probability of paradox in terms that do not depend on $n$. Instead
our results are stated in terms of the statistical distance between $F$ and the closet element in $\F_k(n)$. Thus our result establishes that the only way to avoid
non-transitivity is by being close to the family $\F_k(n)$.

In the following subsections we introduce the probabilistic setup, state our main result, discuss related work and give an outline of the proof.

\subsection{Notation and Quantitative Setup}
We will assume voters vote independently and uniformly at random so each voter chooses one of the $k!$ possible rankings with equal probability. We will write $\P$ for the underlying probability measure
and $\E$ for the corresponding expected value.
In this probabilistic setup, it is natural to measure transitivity as well as how close are two different constitutions.

\begin{itemize}
\item
Given two constitutions $F,G$ on $n$ voters, we denote the statistical distance between $F$ and $G$ by $D(F,G)$, so that:
\[
D(F,G) = \P[F(\sigma) \neq G(\sigma)].
\]
\item
Given a constitution $F$, we write $T(F)$ for the probability that the outcome of $F$ is transitive and $P(F)$ for the probability
that the outcome of $F$ is non-transitive so ($P$ stands for paradox):
\[
T(F) = \P[F(\sigma) \mbox{ is transitive}], \quad P(F) = 1 - T(F).
\]
\end{itemize}

\subsection{Main Result}
In our main result we show that
\begin{Theorem} \label{thm:main}
For every number of alternatives $k \geq 1$ and $\eps > 0$, there exists a $\delta = \delta(\eps)$, such that for every $n \geq 1$, if $F$ is a
constitution on $n$ voters and $k$ alternatives satisfying:
\begin{itemize}
\item
IIA and
\item
$P(F) < \delta$,
\end{itemize}
then there exists $G \in \F_k(n)$ satisfying $D(F,G) < k^2 \eps$.
Moreover, one may take:
\begin{equation} \label{eq:del_main}
\delta = \exp \left(-\frac{C}{\eps^{21}} \right),
\end{equation}
for some absolute constant $0 < C < \infty$.
\end{Theorem}

We therefore obtain the following result stated at the abstract:
\begin{Corollary} \label{cor:main}
For any number of alternatives $k \geq 3$ and $\eps > 0$,
there exists a $\delta = \delta(\eps)$, such that for every $n$, if $F$ is a
constitution on $n$ voters and $k$ alternatives satisfying:
\begin{itemize}
\item
IIA and
\item
$F$ is $k^2 \eps$ far from any dictator, so $D(F,G) > k^2 \eps$ for any dictator $G$,
\item
For every pair of alternatives $a$ and $b$, the probability that $F$ ranks $a$ above $b$ is at least $k^2 \eps$,
\end{itemize}
then the probability of a non-transitive outcome, $P(F)$, is at least $\delta$, where $\delta(\eps)$
may be taken as in~(\ref{eq:del_main}).
\end{Corollary}

\begin{proof}
Assume by contradiction that $P(F) < \delta$. Then  by Theorem~\ref{thm:main} there exists a function $G \in \F_{n,k}$
satisfying $D(F,G) < k^2 \eps$.
Note that for every pair of alternatives $a$ and $b$ it holds that:
\[
\P[G \mbox{ ranks } a \mbox{ above } b] \geq \P[F \mbox{ ranks } a \mbox{ above } b] - D(F,G) > 0.
\]
Therefore for every pair of alternatives there is a positive probability that $G$ ranks $a$ above $b$.
Thus by Theorem~\ref{thm:general_arrow} it follows that $G$ is a dictator which is a contradiction.
\end{proof}

\begin{Remark}
Note that if $G \in \F_k(n)$ and
$F$ is any constitution satisfying $D(F,G) < k^2 \eps$ then $P(F) < k^2 \eps$.
\end{Remark}

\begin{Remark}
The bounds stated in Theorem~\ref{thm:main} and Corollary~\ref{cor:main} in terms of $k$ and $\eps$ is clearly not an optimal one.
We expect that the true dependency has $\delta$ which is some fixed power of $\eps$. Moreover we expect that the bound $D(F,G) < k^2 \eps$ should be improved
to $D(F,G) < \eps$.
\end{Remark}

\subsection{Generalizations and Small Paradox Probability}
Theorem~\ref{thm:main} and Corollary~\ref{cor:main} extend to more general product distributions.
We call a distribution $\mu$ over the permutations of $k$ elements $S(k)$, {\em symmetric} if
$\mu(-\sigma) = \mu(\sigma)$ for all $\sigma \in S(k)$. We will write $\alpha = \alpha(\mu)$ for $\min (\mu(\sigma) : \sigma \in S(k))$. We will write $\P$ and $\E$ for the probability and expected value
according to the product measure $\mu^n$.

\begin{Theorem} \label{thm:gen}
Theorem~\ref{thm:main} and Corollary~\ref{cor:main} extend to the following setup where voters vote independently at random according to a symmetric distribution $\mu$ over the permutations of $k$ elements.
In this setup it suffices to take
\begin{equation} \label{eq:del_gen}
\delta = \exp \left(-\frac{C_1}{\alpha \eps^{C_2(\alpha)}} \right),
\end{equation}

where $0 < C_1(\alpha),C_2(\alpha) < \infty$. In particular one may take
$C_2(\alpha) = 3+1/(2 \alpha^2)$.
\end{Theorem}

The dependency of $\delta$ on $\eps$ in~(\ref{eq:del_main}) and~(\ref{eq:del_gen}) is a bad one. For values of $\eps < O(n^{-1})$
it is possible
to obtain better dependency, where $\delta$ is polynomial in $\eps$. In Section~\ref{sec:almost} we prove the following.

\begin{Theorem} \label{thm:ksmall}
Consider voting on $k$ alternatives where voters vote uniformly at random from $S_k^n$.
Let
\begin{equation} \label{eq:eps_bd_ksmall}
\frac{1}{324} > \eps > 0.
\end{equation}
For every $n$, if $F$ is a
constitution on $n$ voters satisfying:
\begin{itemize}
\item
IIA and
\item
\begin{equation} \label{eq:unif_small}
P(F) < \frac{1}{36} \eps^3 n^{-3},
\end{equation}
\end{itemize}
then there exists $G \in \F_3(n)$ satisfying $D(F,G) \leq 10 k^2 \eps$.
If each voter follows a symmetric voting distribution then with minimal probability $\alpha$ then the same statement holds where~(\ref{eq:eps_bd_ksmall}) is replaced with $\alpha^2/9 > \eps > 0$ and~(\ref{eq:unif_small}) is replaced with
\[
P(F) < \alpha^2 \eps^{-\frac{1}{2 \alpha}} n^{-3}.
\]
\end{Theorem}

\subsection{Related Work}
The first attempt at getting a quantitative version of Arrow's theorem is Theorem 1.2 in a beautiful paper by Kalai~\cite{Kalai:02} which we state in our notation as follows.
\begin{Theorem} \label{thm:kalai}
There exists a $K > 0$ such that the following holds:
Consider voting on $k=3$ alternatives where voters vote uniformly at random from $S_3^n$.
Assume $F$ is a balanced constitution, i.e., for every pair $a,b,$ of alternatives, it holds that the probability that $F$ ranks $a$ above $b$ is exactly $1/2$.
Then if $P(F) < \eps$, then $D(F,G) < K \eps$ for some dictator $G$.
\end{Theorem}

Comparing Kalai's result to Theorem~\ref{thm:main} we see that
\begin{itemize}
\item Kalai obtains better dependency of $\delta$ in terms of $\eps$.
\item Kalai's result holds only for $k=3$ alternatives, while ours hold for any number of alternatives.
\item Kalai's result holds only when $F$ is {\em balanced} while ours hold for all $F$.
\end{itemize}
The approach of~\cite{Kalai:02} is based on "direct" manipulation of the Fourier expression for probability of paradox. A number of unsuccessful attempts (including by the author of the current paper) have been made to extend this approach to a more general setup without assuming balance of the functions and to larger number of alternatives.

A second result of~\cite{Kalai:02} proves that for balanced functions which are transitive the probability of a paradox is bounded away from zero. Transitivity is a strong assumption roughly meaning that all voters have the same power. We do not assume transitivity in the current paper.
A related result~\cite{MoOdOl:05,MoOdOl:09} proved a conjecture of Kalai showing that among all balanced low influence functions, majority minimizes the probability of a paradox.
The low influence condition is weaker than transitivity,but still requires that no single voter has strong influence on the outcome of the vote.



 Keller~\cite{Keller:09} extended some of Kalai's result to symmetric distributions (still under the balance assumption). Keller~\cite{Keller:09} also provides lower bounds on the probability
  of a paradox in the case the functions are monotone and balanced.

 We want to note of some natural limitation to the approach taken in~\cite{Kalai:02} and~\cite{Keller:09} which is based on "direct" analysis of the probability of a paradox in terms of the Fourier expansion. First, this approach does not provide a proof of Arrow theorem nor does it ever use it (while our approach does). Second, it is easy to see that one can get small paradox probability by looking at constitutions on $3$ alternatives which almost always rank one candidates at the top. Thus a quantitative version of Arrow theorem cannot be stated just in terms of distance to a dictator. Indeed an example in~\cite{Keller:09} (see Theorem 1.2) implies that for non-balanced functions the probability of a paradox cannot be related in a linear fashion to the distance from dictator or to other functions in $\F_3(n)$.


As noted in~\cite{Kalai:02}, there is an interesting connection between quantitative Arrow statements and the concept of testing introduced in \cite{RubinfeldSudan:96,GoGoRo:96} which was studied and used extensively since.
Roughly speaking a property of functions is testable if it is possible to perform a randomized test for the property such that if the probability that the function passes the test is close to $1$, then the function has to be close to a function with the property (say in the hamming distance).
In terms of testing, our result states that among all functions satisfying the IIA property, the Transitivity property is testable.
Moreover, the natural test "works": i.e., in order to test for transitivity, one can pick a random input and check if the outcome is transitive.

We finally want to note that the special case of the quantitative Arrow theorem proved by Kalai~\cite{Kalai:02} for balanced functions has been used to derive the first quantitative version of the Gibbard-Satterthwaite Theorem~\cite{Gibbard:73,Satterthwaite:75} in~\cite{FrKaNi:08}. The results of~\cite{FrKaNi:08} are limited in the sense that they require neutrality and apply only to 3 candidates. It is interesting to explore if the full quantitative version of Arrow theorem proven here will allow to obtain stronger quantitative version of the Gibbard-Satterthwaite Theorem.

\subsection{Proof Ideas}
We first recall the notion of {\em influence} of a voter. Recall that for $f : \{-1,1\}^n \to \{-1,1\}$, the influence of
voter $1 \leq i \leq n$ is given by
\[
I_i(f) = \P[f(X_1,\ldots,X_{i-1},0,X_{i+1},\ldots,X_n) \neq f(X_1,\ldots,X_{i-1},1,X_{i+1},\ldots,X_n)],
\]
where $X_1,\ldots,X_n$ are distributed uniformly at random.
The notion of influence is closely related to the notion of {\em pivotal} voter which was introduced in Barabera's proof of Arrow's Theorem~\cite{Barbera:80}.
Recall that voter $i$ is pivotal for $f$ at $x$
if $f(x_1,\ldots,x_{i-1},1,x_{i+1},\ldots,x_n) \neq f(x_1,\ldots,x_{i-1},-1,x_{i+1},\ldots,x_n)$. Thus the influence of voter $i$
is the expected probability that voter $i$ is pivotal.

We discuss the main ideas of the proof for the case $k=3$. By the IIA property that pairwise preference $(a>b),(b>c)$ and $(c>a)$ are decided by three different functions $f,g$ and $h$ depending on the pairwise preference of the individual voters.
\begin{itemize}
\item
The crucial and novel step is showing that for every $\eps > 0$, there exists $\delta > 0$, such that
if two different voters $i \neq j$ satisfy $I_i(f) > \eps$ and $I_j(g) > \eps$, then the probability of a non-transitive outcome is at least $\delta = \eps^C$, for some $C > 0$.
The proof of this step uses and generalizes the results of~\cite{MORSS:06}, which are based on inverse-hyper-contractive estimates~\cite{Borell:85}. We show that if $I_i(f) > \eps$ and $I_j(g) > \eps$ then with probability at least $\eps^C$, over all voters but $i$ and $j$, the restricted $f$ and $g$, have $i$ and $j$ pivotal. 
 We show how this implies that with probability $\eps^C$ we may chose the rankings of $i$ and $j$, leading to a non-transitive outcome. And therefore the probability of a paradox is at least $\eps^C/36$. This step may be viewed as a quantitative version of a result by Barbera~\cite{Barbera:80}. The main step in Barbera's proof of Arrow theorem is proving that if two distinct voters are pivotal for two different pairwise preferences that the constitution has a non-rational outcome.

\item
The results above suffice to establish a quantitative Arrow theorem for $\eps = O(n^{-1})$.
This follows from the fact that all influences of a function are bounded by $\eps n^{-1}$ then the function is $O(\eps)$ close to a constant function. The probability of paradox obtained here is of order $\eps^C$.

\item
Next, we show that the statement of the theorem holds when $n$ is large and all functions $f,g,h$ are symmetric threshold functions. Note that in this case, since symmetric thresholds functions are low influence functions, the conclusion of the theorem reads: if non of the alternatives is ranked at top/bottom with probability $\geq 1-\eps$, then the probability of a paradox is at least $\delta$.

\item
Using the Majority is stablest result~\cite{MoOdOl:09} (see also~\cite{MoOdOl:05}) in the strong form proven in~\cite{Mossel:09} (see also~\cite{Mossel:08}) we extend the result above as long as for any pair of functions say $f,g$ there exist no variable for which {\em both} $I_i(f)$ and $I_i(g)$ is large.

\item
The remaining case is where there exists a single voter $i$, such that $I_i(f)$ is large for at least two of the functions and all other variables have low influences. By expanding the paradox probability in terms of the $6$ possible ranking of voter $i$ and using the previous case, we obtain the conclusion of the theorem, i.e., that in this case either there is a non-negligible probability of a paradox, or the function close to a dictator function on voter $i$.
\end{itemize}

Some notation and preliminaries are given in Section~\ref{sec:pre}. The proof for the case where two different functions have two different influential voters is given in Section~\ref{sec:two}. This already allows to establish a quantitative Arrow theorem in the
case where the functions is very close to an element of $\F_k(n)$ in Section~\ref{sec:almost}. 
The proof of the Gaussian Arrow Theorem is given in Section~\ref{sec:gauss}. Applying "strong" non-linear invariance the result is obtained for low influence functions in Section~\ref{sec:low}. The result with one influential variable is the derived in Section~\ref{sec:one}. The proof of the main result for $3$ alternatives is then given in Section~\ref{sec:3}. Section~\ref{sec:k} concludes the proof by deriving the proof for any number of alternatives. The combinatorial Theorem~\ref{thm:general_arrow} is proven in Section~\ref{sec:general_arrow}. Section~\ref{sec:sym} provides the adjustment of the proofs needed to obtain the results for symmetric distributions.

\subsection{Acknowledgement}
Thanks to Marcus Issacson and Arnab Sen for interesting discussions. Thanks to Salvador Barbera for helpful comments on a manuscript of the paper.  

\section{Preliminaries} \label{sec:pre}
For the proof we introduce some notation and then follow the steps above.
\subsection{Some Notation}
The following notation will be useful for the proof.
A social choice function is a function from a profile on $n$ permutation, i.e., an element of $S(k)^n$ to a binary decision for every pair of alternatives which one is preferable. The set of pairs of candidates is nothing but $k \choose 2$. Therefore a social choice function is a map $F : S(k)^n \to \{-1,1\}^{k \choose 2}$ where $F(\sigma) =
(h^{a>b}(\sigma) : \{a,b\} \in {k \choose 2})$
means
\[
F \mbox{ ranks } a \mbox{ above } b \mbox{ if } h^{a>b}(\sigma) = 1, \quad
F \mbox{ ranks } b \mbox{ above } a \mbox{ if } h^{a>b}(\sigma) = -1.
\]
We will further use the convention that $h^{a>b}(\sigma) = -h^{b>a}(\sigma)$.

The binary notation above is also useful to encode the individual preferences $\sigma(1),\ldots,\sigma(n)$ as follows.
Given $\sigma = \sigma(1),\ldots,\sigma(n)$ we define binary vectors $x^{a>b} = x^{a>b}(\sigma)$ in the following manner:
\[
 x^{a>b}(i) = 1, \quad \mbox{if voter } i \mbox{ ranks } a \mbox{ above } b; \quad
 x^{a>b}(i) = -1, \quad \mbox{if voter } i \mbox{ ranks } a \mbox{ above } b
\]


The IIA condition implies that the pairwise preference between any pair of outcomes depends only on the individual pairwise preferences. Thus, if $F$ satisfies the IIA property then there exists functions $f^{a>b}$ for every pair of candidates $a$ and $b$ such that
\[
F(\sigma) = ((f^{a>b}(x^{a>b}) : \{a,b\} \in {k \choose 2})
\]



We will also consider more general distributions over $S(k)$. We call a distribution $\mu$ on $S(k)$ {\em symmetric} if
$\mu(-\sigma) = \mu(\sigma)$ for all $\sigma \in S(k)$. We will write $\alpha = \alpha(\mu)$ for $\min (\mu(\sigma) : \sigma \in S(k))$.

\subsection{The Correlation Between $x^{a>b}$ and $x^{b>c}$}
For some of the derivations below will need the correlations between the random variables $x^{a>b}(i)$ and $x^{b>c}(i)$.
We have the following easy fact:
\begin{Lemma} \label{lem:cor_vote}
Assume that voters vote uniformly at random from $S(3)$. Then:
\begin{enumerate}
\item
For all $i \neq j$ and all $a,b,c,d$ the variables $x^{a>b}(i)$ and $x^{c>d}(j)$ are independent.
\item
If $a,b,c$ are distinct then $\E[x^{a>b}(i) x^{b>c}(i)] = -1/3$.
\end{enumerate}
\end{Lemma}
For the proof of part 2 of the Lemma, note that the expected value depends only on the distribution over the rankings of $a,b,c$ which is uniform. It thus suffices to consider the case $k=3$. In this case there are $4$ permutations where $x^{a>b}(i) = x^{b>c}(i)$ and
$2$ permutations where $x^{a>b}(i) \neq x^{b>c}(i)$.



We will also need the following estimate
\begin{Lemma} \label{lem:cor_vote2}
Assume that voters vote uniformly at random from $S(3)$. Let $f = x^{c>a}$ and
let $(Tf)(x^{a>b},x^{b>c}) = \E[f | x^{a>b},x^{b>c}]$. Then
\[
| Tf |_2 = 1/\sqrt{3}.
\]
\end{Lemma}

\begin{proof}
There are two permutations where $x^{a>b},x^{b>c}$ determine $x^{c>a}$. For all other permutations $x^{c>a}$ is equally likely to be $-1$ and $1$ conditioned on $x^{a>b}$ and $x^{b>c}$. We conclude that $| Tf |_2^2 = 1/3$ and therefore
$|Tf|_2 = 1/\sqrt{3}$.
\end{proof}




\subsection{Inverse Hyper-contraction and Correlated Intersections Probabilities}
We will use some corollaries of the inverse hyper-contraction estimates proven by Borell~\cite{Borell:82}.
The following corollary is from~\cite{MORSS:06}.

\begin{Lemma} \label{lem:inv_hyp_org}
Let $x,y \in \{-1,1\}^n$ be distributed uniformly and $(x_i,y_i)$ are independent.
Assume that $\E[x(i)] = \E[y(i)] = 0$ for all $i$ and that
$\E[x(i) y(i)] = \rho \geq 0$. Let $B_1, B_2 \subset \{-1,1\}^n$ be two sets and assume that
\[
\P[B_1] \geq e^{-\alpha^2}, \quad \P[B_2] \geq e^{-\beta^2}.
\]
Then:
\[
\P[x \in B_1, y \in B_2] \geq \exp(-\frac{\alpha^2+\beta^2+2 \rho \alpha \beta}{1-\rho^2}).
\]
\end{Lemma}

We will need to generalize the result above to negative $\rho$ and further to different $\rho$ value for different bits.

\begin{Lemma} \label{lem:inv_hyp}
Let $x,y \in \{-1,1\}^n$ be distributed uniformly and $(x_i,y_i)$ are independent.
Assume that $\E[x(i)] = \E[y(i)] = 0$ for all $i$ and that
$|\E[x(i) y(i)]| \leq \rho$. Let $B_1, B_2 \subset \{-1,1\}^n$ be two sets and assume that
\[
\P[B_1] \geq e^{-\alpha^2}, \quad \P[B_2] \geq e^{-\beta^2}.
\]
Then:
\[
\P[x \in B_1, y \in B_2] \geq \exp(-\frac{\alpha^2+\beta^2+2 \rho \alpha \beta}{1-\rho^2}).
\]
In particular if $\P[B_1] \geq \eps$ and $\P[B_2] \geq \eps$, then:
\begin{equation} \label{eq:two_small}
\P[x \in B_1, y \in B_2] \geq \eps^{\frac{2}{1-\rho}}.
\end{equation}
\end{Lemma}

\begin{proof}
Take $z$ so that $(x_i,z_i)$ are independent and $\E[z_i] = 0$ and $\E[x_i z_i] = \rho$.
It is easy to see there exists $w_i$ independent of $x,z$ with s.t. the joint distribution of
$(x,y)$ is the same as $(x,z \cdot w)$, where $z \cdot w = (z_1 w_1,\ldots,z_n w_n)$.
Now for each fixed $w$ we have that
\[
\P[x \in B_1, z \cdot w \in B_2] =
\P[x \in B_1, z \in w \cdot B_2] \geq \exp(-\frac{\alpha^2+\beta^2+2 \rho \alpha \beta}{1-\rho^2}),
\]
where $w \cdot B_2 = \{w \cdot w' : w' \in B_2 \} $.
Therefore taking expectation over $w$ we obtain:
\[
\P[x \in B_1, y \in B_2] = \E \P[x \in B_1, z \cdot w \in B_2] \geq \exp(-\frac{\alpha^2+\beta^2+2 \rho \alpha \beta}{1-\rho^2})
\]
as needed.
The conclusion~(\ref{eq:two_small}) follows by simple substitution (note the difference with Corollary 3.5 in~\cite{MORSS:06} for sets of equal size which is a typo).
\end{proof}

Applying the CLT and using~\cite{Borell:85} one obtains the same result for Gaussian random variables.
\begin{Lemma} \label{lem:inv_hyp_gauss}
Let $N,M$ be $N(0,I_n)$ with $(N(i),M(i))_{i=1}^n$ independent.
Assume that
$|\E[N(i) M(i)]| \leq \rho$. Let $B_1, B_2 \subset \R^n$ be two sets and assume that
\[
\P[B_1] \geq e^{-\alpha^2}, \quad \P[B_2] \geq e^{-\beta^2},
\]
Then:
\[
\P[N \in B_1, M \in B_2] \geq \exp(-\frac{\alpha^2+\beta^2+2 \rho \alpha \beta}{1-\rho^2}).
\]
In particular if $\P[B_1] \geq \eps$ and $\P[B_2] \geq \eps$, then:
\begin{equation} \label{eq:two_small_gauss}
\P[N \in B_1, M \in B_2] \geq \eps^{\frac{2}{1-\rho}}.
\end{equation}
\end{Lemma}

\begin{proof}
Fix the values of $\alpha$ and $\beta$ and assume without loss of generality that $\max_i |\E[N(i) M(i)]|$ is obtained for $i=1$.
Then by~\cite{Borell:85} (see also~\cite{Mossel:09}), the minimum of the quantity $\P[N \in B_1, M \in B_2]$ under
the constraints on the measures given by $\alpha$ and $\beta$ is obtained in one dimension, where $B_1$ and $B_2$ are intervals
$I_1, I_2$. Look at random variables $x(i), y(i)$, where $\E[x(i)] = \E[y(i)] = 0$ and
$\E[x(i) y(i)] =  \E[M_1 N_1]$. Let $X_n = n^{-1/2} \sum_{i=1} x^{a>b}(i)$ and $Y_n = n^{-1/2} \sum_{i=1} x^{a>b}(i)$.
Then the CLT implies that
\[
\P[X_n \in I_1] \to \P[N_1 \in B_1], \quad
\P[Y_n \in I_2] \to \P[M_1 \in B_2],
\]
and
\[
\P[X_n \in I_1, Y_n \in I_2] \to \P[N_1 \in B_1, M_1 \in B_2].
\]
The proof now follows from the previous lemma.
\end{proof}

\section{Two Influential Voters} \label{sec:two}
We begin the proof of Arrow theorem by considering the case of $3$ candidates named $a,b,c$ and two influential voters named $1$ and $2$.
Note that for each voter there are $6$ legal values for $(x_i^{a>b},x_i^{b>c},x_i^{c>a})$. These are all vector different from $(-1,-1,-1)$ and $(1,1,1)$.
Similarly constitution given by $f^{a>b},f^{b>c}$ and $f^{c>a}$ has a non-transitive outcome if and only if
\[
(f^{a>b}(x^{a>b}),f^{b>c}(x^{b>c}),f^{c>a}(x^{c>a})) \in \{(-1,-1,-1),(1,1,1)\}.
\]

\subsection{Two Pivots Imply Paradox}
We will use the following Lemma which as kindly noted by Barbera was first proven in his paper~\cite{Barbera:80}. 

\begin{Proposition} \label{prop:32}
Consider a social choice function on $3$ candidates $a,b$ and $c$ and two voters denoted $1$ and $2$. Assume that the social choice function satisfies that IIA condition and that voter $1$ is pivotal for $f^{a>b}$ and voter $2$ is pivotal for $f^{b>c}$.
Then there exists a profile for which
$(f^{a>b}(x^{a>b}),f^{b>c}(x^{b>c}),f^{c>a}(x^{c>a}))$ is non-transitive.
\end{Proposition}

For completeness we provide a proof using the language of the current paper (the proof of~\cite{Barbera:80} like much of the literature on Arrow's theorem uses binary relation notation).

\begin{proof}
Since voter $1$ is pivotal for $f^{a>b}$ and voter $2$ is pivotal for $f^{b>c}$
there exist $x,y$ such that
\[
f^{a>b}(0,y) \neq f^{a>b}(1,y), \quad f^{b>c}(x,0) \neq f^{b>c}(x,1).
\]
Look at the profile where
\[
x^{a>b} = (x^{\ast},y), \quad
x^{b>c} = (x,y^{\ast}), \quad
x^{c>a} = (-x,-y).
\]
We claim that for all values of $x^{\ast},y^{\ast}$ this correspond to transitive rankings of the two voters. This follows from the fact that neither $(x^{\ast},x,-x)$ nor $(y,y^{\ast},-y)$ belong to the set $\{(1,1,1),(-1,-1,-1)\}$. Note furthermore we may chose $x^{\ast}$ and $y^{\ast}$ such that
\[
f^{c>a}(-x,-y) = f^{a>b}(x^{\ast},y) = f^{b>c}(x,y^{\ast}).
\]
We have thus proved the existence of a non-transitive outcome as needed.
\end{proof}

\subsection{Two influential Voters Implies Joint Pivotality}
Next we establish the following result.
\begin{Lemma} \label{lem:two_inf}
Consider a social choice function on $3$ candidates $a,b$ and $c$ and $n$ voters denoted $1,\ldots,n$. Assume that the social choice function satisfies that IIA condition and that voters vote uniformly at random. Assume further that
$I_1(f^{a>b}) > \eps$ and $I_1(f^{b>c}) > \eps$.
Let
\[
B = \{\sigma : 1 \mbox{ is pivotal for } f^{a>b}(x^{a>b}(\sigma)) \mbox{ and } 2 \mbox{ is pivotal for } f^{b>c}(x^{b>c}(\sigma)) \}.
\]
Then
\[
\P[B] \geq \eps^3.
\]
\end{Lemma}

\begin{proof}
Let
\[
B_1 = \{\sigma : 1 \mbox{ is pivotal for } f^{a>b} \}, \quad B_2 = \{\sigma : 2 \mbox{ is pivotal for } f^{b>c} \}.
\]
Then $\P[B_1] = I_1(f^{a>b}) > \eps$ and $\P[B_2] = I_2(f^{b>c}) > \eps$, and our goal is to obtain a bound on $\P[B_1 \cap B_2]$.
Note that the event $B_1$ is determined by $x^{a>b}$ and the event $B_2$ is determined by $x^{b>c}$.
Further by Lemma~\ref{lem:cor_vote}
it follows that
$\E[x^{a<b}(i)] = \E[x^{b>c}(i)] = 0$ and $|\E[x^{a>b}(i) x^{b>c}(i)]| = 1/3$
The proof now follows from Lemma~\ref{lem:inv_hyp}.
\end{proof}

\subsection{Two Influential Voters Imply Non-Transitivity}
We can now prove the main result of the section.
\begin{Theorem} \label{thm:two_inf}
Let $k \geq 3$ and $\eps > 0$. Consider the uniform voting model on $S(k)$.
Let $F$ be a
constitution on $n$ voters satisfying:
\begin{itemize}
\item
IIA and
\item
There exists three distinct alternatives $a,b$ and $c$ and two distinct voters $i$ and $j$ such that
\[
I_i(f^{a>b}) > \eps, \quad I_j(f^{b>c}) > \eps.
\]
\end{itemize}
then $P(F) > \frac{1}{36} \eps^3$.

\end{Theorem}

\begin{proof}
We look at $F$ restricted to rankings of $a,b$ and $c$. Note that in the uniform case each permeation has probability $1/6$ 
Without loss of generality assume that $i=1$ and $j=2$ and consider first the case of the uniform distribution over rankings.
let $B$ be the event from Lemma~\ref{lem:two_inf}. By the lemma we have $\P[B] \geq \eps^3$.
Note that if $\sigma \in S(3)^n$ satisfies that $\sigma \in B$, then fixing $\sigma(3),\ldots,\sigma(n)$ we may apply Proposition~\ref{prop:32} to conclude that there are values of $\sigma^{\ast}(1)$ and $\sigma^{\ast}(2)$
leading to a non-transitive outcome.
Therefore:
\[
\P[P(F)] \geq \P[(\sigma^{\ast}(1),\sigma^{\ast}(2),\sigma(3),\ldots,\sigma(n)) : \sigma \in B] \geq \frac{1}{36} \P[B] \geq
\frac{1}{36} \eps^3.
\]
\end{proof}

\section{Arrow Theorem for Almost Transitive Functions} \label{sec:almost}
In this section we prove a quantitative Arrow Theorem in the case where the probability of a non-transitive outcome is inverse polynomial in $n$. In this case it is possible to obtain an easier quantitative proof which does not rely on invariance. We will use the following easy and well known Lemma.

\begin{Lemma} \label{lem:inf_sum}
Let $f : \{-1,1\}^n \to \{-1,1\}$ and assume $I_i(f) \leq \eps n^{-1}$ for all $i$.
Then there exist a constant function $s \in \{-1,1\}$ such that
$D(f,s) \leq 2 \eps$.

Similarly, let $f : \{-1,1\}^n \to \{-1,1\}$ and assume $I_i(f) \leq \eps n^{-1}$ for all $i \neq j$.
Then there exists a function $g :\{-1,1\} \to \{-1,1\}$ such that
$D(f,g(x_j)) \leq 2 \eps$.
\end{Lemma}

\begin{proof}
For the first claim, use
\begin{equation} \label{eq:inf_sum}
\frac{1}{2} \min(\P[f=1], \P[f = -1]) \leq
\P[f = 1] \P[f = -1] = \Var[f] \leq \sum_{i=1}^n I_i(f) \leq \eps.
\end{equation}

For the second claim assume WLOG that $j=1$. Let $f_1(x_2,\ldots,x_n) = f(1,x_2,\ldots,x_n)$ and
$f_{-1}(x_2,\ldots,x_n) = f(-1,x_2,\ldots,x_n)$. Apply~(\ref{eq:inf_sum})
to chose $s_1$ so that
\[
D(f_1,s_1) \leq \sum_{i > 1} I_i(f_1).
\]
Similarly, let $s_{-1}$ be chosen so that
\[
D(f,s_{-1}) \leq \sum_{i > 1} I_i(f_{-1}).
\]
Let $g(1) = s_1$ and $g(-1) = s_{-1}$. Then:
\[
2 D(f,g) = D(f_1,s_1) + D(f_{-1},s_{-1}) \leq \sum_{i > 1} I_i(f_1) + \sum_{i > 1} I_i(f_{-1}) =
2 \sum_{i > 1} I_i(f) \leq 2 \eps.
\]
The proof follows.
\end{proof}

\begin{Theorem} \label{thm:3small}
Consider voting on $3$ alternatives where voters vote uniformly at random from $S_3^n$.
Let
\begin{equation} \label{eq:eps_bd_3small}
\frac{1}{324} > \eps > 0.
\end{equation}
For every $n$, if $F$ is a
constitution on $n$ voters satisfying:
\begin{itemize}
\item
IIA and
\item
\begin{equation} \label{eq:unif_small3}
P(F) < \frac{1}{36} \eps^3 n^{-3},
\end{equation}
\end{itemize}
then there exists $G \in \F_3(n)$ satisfying $D(F,G) \leq 10 \eps$.


\end{Theorem}

\begin{proof}
We prove the theorem for the uninform case. The proof for the symmetric case is identical.
Let $f^{a>b},f^{b>c},f^{c>a} : \{-1,1\}^n \to \{-1,1\}$ be the three pairwise preference functions.
Let $\eta = \eps n^{-1}$.

Consider three cases:
\begin{enumerate}
\item[I.]
Among the functions $f^{a>b},f^{b>c},f^{c>a}$, there exist two different functions $f$ and $g$ and
$i \neq j$ s.t: $I_i(f) > \eta$ and $I_i(g) > \eta$.
\item[II.]
There exists a voter $i$ such that for all $j \neq i$ and all $f \in \{f^{a>b},f^{b>c},f^{c>a}\}$, it holds that $I_j(f) < \eta$.
\item[III.]
There exists two different functions $f, g \in \{f^{a>b},f^{b>c},f^{c>a}\}$ such that for all $i$ it holds that $I_i(f) < \eta$ and $I_i(g) < \eta$.
\end{enumerate}

Note that each $F$ satisfies one of the three conditions above.
Note further that in case I. we have $P(F) > \frac{1}{36} \eps^3 n^{-3}$ by Theorem~\ref{thm:two_inf} which contradicts the assumption~(\ref{eq:unif_small3}).
So to conclude the proof is suffices to obtain $D(F,G) \leq 10 \eps$ assuming~(\ref{eq:unif_small3}).

In case II. it follows from Lemma~\ref{lem:inf_sum} that there exists functions $g^{a>b}, g^{b>c}$ and $g^{c>a}$ of voter $i$ only such that
\[
D(f^{a>b},g^{a>b}) < 2 \eps, \quad
D(f^{b>c},g^{b>c}) < 2 \eps, \quad
D(f^{c>a},g^{c>a}) < 2 \eps.
\]
Letting $G$ be the constitution defined by the $g$'s we therefore have $D(F,G) \leq 6 \eps$ and
$P(G) \leq P(F) + 6 \eps \leq 9 \eps$.

Furthermore if $9 \eps < \frac{1}{36}$ this implies that $PDX(G) = 0$. So $D(F,F_3(n)) \leq 6 \eps$
which is a contradiction.

In the remaining case III. assume WLOG that $f^{a>b}$ and $f^{b>c}$ have all influences small. By Lemma~\ref{lem:inf_sum}
if follows that $f^{a>b}$ and $f^{b>c}$ are $2 \eps$ far from a constant function.
There are now two subcases to consider. In the first case there exists an $s \in \{ \pm 1 \}$ such that
$D(f^{a>b},s) \leq 2 \eps$ and $D(f^{b>c},-s) \leq 2 \eps$. Note that in the case letting
\[
g^{a>b} = s, \quad g^{b>c} = -s, \quad g^{c>a} = f^{c>a},
\]
and $G$ be the constitution defined by the $g$'s, we obtain that $G \in F_3(n)$ and
$D(F,G) \leq 4 \eps$.

We finally need to consider the case where $D(f^{a>b},s) \leq 2 \eps$ and $D(f^{b>c},s) \leq 2 \eps$
for some $s \in \{ \pm 1 \}$. Let $A(a,b)$ be the set of $\sigma$ where $f^{a>b} = -s$ and similarly for $A(b,c)$ and $A(a,c)$. Then $\P[A(a,b)] \leq 2 \eps$ and $\P[A(a,c)] \leq 2 \eps$. Furthermore by transitivity
\[
\P[A(a,c)] \leq \P[A(a,b)] + \P[A(b,c)] + P(F) \leq 6 \eps.
\]
We thus conclude that $D(f^{c>a},s) \leq 6 \eps$. Letting $g^{a>b} = g^{b>c} = -g^{c>a} = s$ and $G$ the constitution defined by $G$ we have that $D(F,G) \leq 10 \eps$. A contradiction.
The proof follows

\end{proof}

It is now easy to prove Theorem~\ref{thm:ksmall} for the uniform distribution. The adaptations to symmetric distributions will be discussed in Section~\ref{sec:sym}.

\begin{proof}
The proof follows by applying Theorem~\ref{thm:3small} to triplets of alternatives. We give the proof for the uniform case.
Assume $P(F) < \frac{1}{36} \eps^3 n^{-3}$.

Note that if $g_1,g_2 : \{-1,1\}^n \to \{-1,1\}$ are two different function each of which is either
a dictator or a constant function than $D(g_1,g_2) \geq 1/2$. Therefore for all $a,b$ it holds that
$D(f^{a>b},g) < 10 \eps$ for at most one function $g$ which is either a dictator or a constant function.
In case there exists such function we let $g^{a>b} = g$, otherwise, we let $g^{a>b} = f^{a>b}$.

Let $G$ be the social choice function defined by the functions $g^{a>b}$. Clearly:
\[
D(F,G) < 10 {k \choose 2} \eps < 10 k^2 \eps.
\]
The proof would follow if we could show $P(G) = 0$ and therefore $G \in \F_k(n)$.

To prove that $G \in F_k(n)$ is suffices to show that for every set $A$ of three alternatives, it holds that $G_A \in \F_3(n)$. Since $P(F_A) \leq P(F) < \frac{1}{36} \eps^3 n^{-3}$, Theorem~\ref{thm:3small} implies that there exists a function $H_A \in F_3(n)$ s.t. $D(H_A,F_A) < 10 \eps$.
There are two cases to consider:
\begin{itemize}
\item
$H_A$ is a dictator. This implies that $f^{a>b}$ is $10 \eps$ close to a dictator for each $a,b$ and therefore $f^{a>b} = g^{a>b}$ for all pairs $a,b$, so $G_A = H_A \in \F_3(n)$.
\item
There exists an alternative (say $a$) that $H_A$ always ranks at the top/bottom.
In this case we have that $f^{a>b}$ and $f^{c>a}$ are at most $\eps$ far from the constant functions $1$ and $-1$ (or $-1$ and $1$).
The functions $g^{a>b}$ and $g^{c>a}$ have to take the same constant values and therefore again we have that $G_A \in \F_3(n)$.
\end{itemize}

The proof follows.

\end{proof}

\section{The Gaussian Arrow Theorem} \label{sec:gauss}

The next step is to consider a Gaussian version of the problem. The Gaussian version corresponds to a situation
where the functions $f^{a>b},f^{b>c},f^{c>a}$ can only "see" averages of large subsets of the voters.
We thus define a $3$ dimensional normal vector $N$.
The first coordinate of $N$ is supposed to represent the deviation of the number of voters where $a$ ranks above $b$ from the mean.
The second coordinate is for $b$ ranking above $c$ and the last coordinate for $c$ ranking above $a$.

Since averaging maintain the expected value and covariances, we define:
\begin{eqnarray} \label{eq:gauss}
\E[N_1^2] = \E[N_2^2] = \E[N_3^2] = 1,\\ \nonumber
\E[N_1 N_2] = \E[x^{a>b}(1) x^{b>c}(1)] := -1/3,\\ \nonumber
 \E[N_2 N_3] = \E[x^{b>c}(1) x^{c>a}(1)] := -1/3,\\ \nonumber
\E[N_3 N_1] = \E[x^{c>a}(1) x^{a>b}(1)] := -1/3. \nonumber
\end{eqnarray}
We let $N(1),\ldots,N(n)$ be independent copies of $N$.
We write $\CalN = (N(1),\ldots,N(n))$ and for $1 \leq i \leq 3$ we write $\CalN_i = (N(1)_i,\ldots,N(n)_i)$.
The Gaussian version of Arrow theorem states:

\begin{Theorem} \label{thm:arrow_gauss}
For every $\eps > 0$ there exists a $\delta = \delta(\eps) > 0$ such that the following hold.
Let $f_1,f_2,f_3 : \R^n \to \{-1,1\}$.
Assume that for all $1 \leq i \leq 3$ and all $u \in \{-1,1\}$ it holds that
\begin{equation} \label{eq:non_dict_normal}
\P[f_i(\CalN_i) = u, f_{i+1}(\CalN_{i+1}) =-u] \leq 1-\eps
\end{equation}
Then with the setup given in~(\ref{eq:gauss}) it holds that:
\[
\P[f_1(\CalN_1) = f_2(\CalN_2) = f_3(\CalN_3)] \geq \delta.
\]
Moreover, one may take $\delta = (\eps/2)^{18}$.
\end{Theorem}
We note that the negation of condition~(\ref{eq:non_dict_normal}) corresponds to having one of the alternatives at the top/bottom with probability at least $1-\eps$. Therefore the theorem states that unless this is the case, the probability of a paradox is at least $\delta$. Since the Gaussian setup excludes dictator functions in terms of the original vote, this is the result to be expected in this case.

\begin{proof}
We will consider two cases: either all the functions $f_i$ satisfy $|\E f_i| \leq 1-\eps$, or there exists at least one function 
with $|\E f_i| > 1 - \eps$. 

Assume first that there exist a function $f_i$ with $|\E f_i| > 1-\eps$.
Without loss of generality assume that $\P[f_1 = 1] > 1-\eps/2$.
Note that by~(\ref{eq:non_dict_normal}) it follows that $\P[f_2 = 1] > \eps/2$ and $\P[f_3 = 2] > \eps/2$.
By Lemma~\ref{lem:inv_hyp_gauss}, we have $\P[f_2(\CalN_2) = 1, f_3(\CalN_3) = 1] > (\eps/2)^3$.
We now look at the function $g = 1(f_2 = 1, f_3 = 1)$. Let
\[
\CalM_1 = \frac{\sqrt{3}}{2} (\CalN_2 + \CalN_3), \quad
\CalM_2 = \frac{\sqrt{3}}{2 \sqrt{2}} (\CalN_2 - \CalN_3).
\]
Then it is easy to see that $\CalM_2(i)$ is uncorrelated with and therefore independent off $\CalN_1(i),\CalM_1(i)$ for all $i$.
Moreover, for all $i$ the covariance between $\CalM_1(i)$ and $\CalN_1(i)$ is $1/\sqrt{3}$
(this also follows from Lemma~\ref{lem:cor_vote2}) and
$1-1/\sqrt{3} > 1/3$.
We may now apply Lemma~\ref{lem:inv_hyp_gauss} with the vectors
\[
(\CalN_1(1),\ldots,\CalN_1(n),Z_1,\ldots,Z_n), \quad
(\CalM_1(1),\ldots,\CalM_1(n),\CalM_2(1),\ldots,\CalM_2(n)),
\]
where $Z=(Z_1,\ldots,Z_n)$ is a normal Gaussian vector independent of anything else. We obtain:
\[
\P[f_1(\CalN_1) = 1, f_2(\CalN_2) = 1, f_3(\CalN_3) = 1] =
\P[f_1(\CalN_1,Z) = 1, g(\CalM_1,\CalM_2) = 1] \geq ((\eps/2)^3)^{\frac{2}{1/3}} \geq (\eps/2)^{18}.
\]

 We next consider the case where all functions satisfy $|\E f_i| \leq 1-\eps$. In this case at least two of the functions obtain the same value with probability at least a $1/2$. Let's assume that $\P[f_1 = 1] \geq 1/2$ and
$\P[f_2 = 1] \geq 1/2$. Then by Lemma~\ref{lem:inv_hyp_gauss} we obtain that
\[
\P[f_1 = 1, f_2 = 1] \geq 1/8.
\]
Again we define $g = 1(f_1 = 1, f_2 = 1)$. Since $\P[f_3 = 1] > \eps/2$, we may apply Lemma~\ref{lem:inv_hyp_gauss} and obtain that:
\[
\P[f_1 = 1, f_2 = 1, f_3 = 1] =
\P[f_1 = 1, g = 1] \geq (\eps/2)^3.
\]
This concludes the proof.

\end{proof}

\section{Arrow Theorem for Low Influence Functions} \label{sec:low}
Our next goal is to apply Theorem~\ref{thm:arrow_gauss} along with invariance in order to obtain Arrow theorem for low influence functions. Non linear invariance principles were proven in~\cite{Rotar:79} and latter in~\cite{MoOdOl:09} and~\cite{Mossel:09}. We will use the two later results which have quantitative bounds in terms of the influences. The proof for uniform voting distributions follows in a straightforward manner from Theorem~\ref{thm:arrow_gauss}, Kalai's formula and the Majority is Stablest (MIST) result in the strong form stated at~\cite{DiMoRe:06,Mossel:09} where it is allowed that for each variable one of the functions has high influence.
  The proof follows since Kalai's formula allows to write the probability of a paradox as sum of correlation terms between pairs of function and each correlation factor is asymptotically minimized by symmetric monotone threshold functions. Therefore the overall expression is also minimized by symmetric monotone threshold functions. However, Theorem~\ref{thm:arrow_gauss} provides a lower bound on the probability of paradox for symmetric threshold functions so the proof follows.
  The case of symmetric distributions is much more involved and will be discussed in subsection~\ref{subsec:low}.


We finally note that the application of invariance is the step of the proof where $\delta$ becomes very small (more than exponentially small in $\eps$, instead of just polynomially small). A better error estimate in invariance principles in terms of influences
will thus have a dramatic effect on the value of $\delta$.

\subsection{Arrow's theorem for low influence functions.}
We first recall the following result from Kalai~\cite{Kalai:02}.

\begin{Lemma} \label{lem:kalai}
Consider a constitution $F$ on $3$ voters satisfying IIA and let $F$ be given by
$f^{a>b},f^{b>c}$ and $f^{c>a}$. Then:
\begin{equation} \label{eq:PF}
P(F) = \frac{1}{4} \left( 1 + \E[f^{a>b}(x^{a>b})f^{b>c}(x^{b>c})] + \E[f^{b>c}(x^{b>c})f^{c>a}(x^{c>a})] +
\E[f^{c>a}(x^{c>a})f^{a>b}(x^{a>b})] \right)
\end{equation}
\end{Lemma}

\begin{proof}
Let $s : \{-1,1\}^3 \to \{0,1\}$ be the indicator function of the set $\{(1,1,1),(-1,-1,-1)\}$.
Recall that the outcome of $F$ is non-transitive iff
\[
s(f^{a>b}(x^{a>b}),f^{b>c}(x^{b>c}),f^{c>a}(x^{c>a})) = 1.
\]
Moreover $s(x,y,z) = 1/4(1 + xy + yz + zx)$.
The proof follows.
\end{proof}
\begin{Theorem} \label{thm:arrow_low_cross_unif}
For every $\eps > 0$ there exists a $\delta(\eps) > 0$ and a $\tau(\delta) > 0$
such that the following holds.
Let $f_1,f_2,f_3 : \{-1,1\}^n \to \{-1,1\}$.
Assume that for all $1 \leq i \leq 3$ and all $u \in \{-1,1\}$ it holds that

\begin{equation} \label{eq:non_top_bot_cross_unif}
\P[f_i = u, f_{i+1} = -u] \leq 1-2\eps
\end{equation}
and for all $j$ it holds that
\begin{equation} \label{eq:cross_inf_unif}
|\{ 1 \leq i \leq 3 : I_j(f_i) > \tau \}| \leq 1.
\end{equation}
Then it holds that
\[
\P(f_1,f_2,f_3) \geq \delta.
\]
Moreover, assuming the uniform distribution, one may take:
\[
\delta = \frac{1}{8}(\eps/2)^{20}, \quad \tau = \tau(\delta),
\]
where
\[
\tau(\delta) := \delta^{C \frac{\log(1/\delta)}{\delta}},
\]
for some absolute constant $C$.
\end{Theorem}

\begin{proof}
Let $g_1,g_2,g_3 : \R \to \{-1,1\}$ be of the form $g_i = \sgn(x-t_i)$, where $t_i$ is chosen so that
$\E[g_i] = \E[f_i]$ (where the first expected value is according to the Gaussian measure). Let $N_1,N_2,N_3 \sim N(0,1)$ be jointly
Gaussian with $\E[N_i N_{i+1}] = -1/3$.
From Theorem~\ref{thm:arrow_gauss} it follows that:
\[
P(g_1,g_2,g_3) > 8 \delta,
\]
and from the Majority is Stablest theorem as stated in Theorem 6.3 and Lemma 6.8 in~\cite{Mossel:09}, it follows that by choosing
$C$ in the definition of $\tau$ large enough, we have:
\[
\E[f_1(x^{a>b})f_2(x^{b>c})] \geq \E[g_1(N_1) g_2(N_2)] - \delta, \quad
\E[f_2(x^{a>b})f_3(x^{b>c})] \geq \E[g_2(N_1) g_3(N_2)] - \delta,
\]
\[
\E[f_3(x^{a>b})f_1(x^{b>c})] \geq \E[g_3(N_1) g_1(N_2)] - \delta.
\]
From~(\ref{eq:PF}) and~(\ref{eq:PG}) it now follows that:
\[
P(f_1,f_2,f_3) \geq P(g_1,g_2,g_3) - 3 \delta/4 > 7 \delta,
\]
as needed.
\end{proof}

\section{One Influential Variable} \label{sec:one}
The last case to consider is where there is a single influential variable.
This case contains in particular the case of the dictator function. Indeed, our goal in this section
will be to show that if there is a single influential voter and the probability of an irrational outcome is small, then
the function must be close to a dictator function or to a function where one of the alternatives is always ranked at the
bottom (top).

\begin{Theorem} \label{thm:arrow_one_inf}
Consider the voting model with three alternatives and either uniform votes and $\alpha=1/6$, 
For every $\eps > 0$ there exists a $\delta(\eps) > 0$ and a $\tau(\delta) > 0$ such that the following holds.
Let $f_1,f_2,f_3 : \{-1,1\}^n \to \{-1,1\}$ and let $F$ be the social choice function defined by letting
$f^{a>b} = f_1, f^{b>c} = f_2$ and $f^{c>a} = f_3$.
Assume that for all $1 \leq i \leq 3$ and $j > 1$ it holds that
\begin{equation} \label{eq:inf5}
I_j(f_i) < \alpha \tau.
\end{equation}
Then either
\begin{equation} \label{eq:paradox5}
\P(f_1,f_2,f_3) \geq \alpha \delta,
\end{equation}
or there exists a function $G \in \F_3(n)$ such that $D(F,G) \leq 9 \eps$.
Moreover, assuming the uniform distribution, one may take:
\[
\delta = (\eps/2)^{20}, \quad \tau = \tau(\delta).
\]
\end{Theorem}

\begin{proof}
Consider the functions $f_i^b$ for $1 \leq i \leq 3$ and $b \in \{-1,1\}$ defined by
\[
f_i^b(x_2,\ldots,x_n) = f_i(b,x_2,\ldots,x_n).
\]
Note that for all $b \in \{-1,1\}$, for all $1 \leq i \leq 3$ and for all $j > 1$ it holds that $I_j(f_i^{b_i}) < \tau$
and therefore we may apply 
Theorem~\ref{thm:arrow_low_cross_unif}.
We obtain that for every
$b=(b_1,b_2,b_3) \notin \{(1,1,1),(-1,-1,-1)\}$ either:
\begin{equation} \label{eq:small_paradox}
\P(f_1^{b_1},f_2^{b_2},f_3^{b_3}) \geq \delta,
\end{equation}
or there exist a $u(b,i) \in \{-1,1\}$ and an $i = i(b)$ such that
\begin{equation} \label{eq:small_constant}
\min(\P[f_i^{b_i} = u(b,i)], \P[f_{i+1}^{b_{i+1}} = -u(b,i)]) \geq 1-3\eps.
\end{equation}
Note that if there exists a vector $b = (b_0,b_1,b_2) \notin \{(1,1,1),(-1,-1,-1)\}$ for which~(\ref{eq:small_paradox}) holds
then~(\ref{eq:paradox5}) follows immediately.

It thus remains to consider the case where~(\ref{eq:small_constant}) holds for all $6$ vectors $b$.
In this case we will define new functions $g_i$ as follows. We let $g_i(b,x_2,\ldots,x_n) = u$ if
$\P[f_i^{b_i} = u] \geq 1-3\eps$ for $u \in \{-1,1\}$ and $g_i(b,x_2,\ldots,x_n) = f_i(b,x_2,\ldots,x_n)$ otherwise.
We let $G$ be the social choice function defined by $g_1,g_2$ and $g_3$.
From~(\ref{eq:small_constant}) it follow that for every $b = (b_0,b_1,b_2) \notin \{(1,1,1),(-1,-1,-1)\}$ there exists two functions
$g_i,g_{i+1}$ and a value $u$ s.t. $g_i(b_i,x_2,\ldots,x_n)$ is the constant function $u$ and $g_{i+1}(b_{i+1},x_2,\ldots,x_n)$
is the constant function $-u$. So
\[
P(g_1,g_2,g_3) = \P[(g_1,g_2,g_3) \in \{(1,1,1),(-1,-1,-1)\}] = 0,
\]
and therefore $G \in \F_3(n)$. It is further easy to see that $D(f_i,g_i) \leq 3 \eps$ for all $i$ and therefore:
\[
D(F,G) \leq D(f_1,g_1) + D(f_2,g_2) + D(f_3,g_3) \leq 9 \eps.
\]
The proof follows.
\end{proof}

\section{Quantitative Arrow Theorem for $3$ Candidates} \label{sec:3}
We now prove a quantitative version of Arrow theorem for $3$ alternatives.

\begin{Theorem} \label{thm:3}
Consider voting on $3$ alternatives where voters vote uniformly at random from $S_3^n$.
Let $\eps > 0$. Then there exists a $\delta = \delta(\eps)$, such that for every $n$, if $F$ is a
constitution on $n$ voters satisfying:
\begin{itemize}
\item
IIA and
\item
$P(F) < \delta$,
\end{itemize}
then there exists $G \in \F_3(n)$ satisfying $D(F,G) < \eps$.
Moreover, one can take
\begin{equation} \label{eq:del33}
\delta = \exp \left(-\frac{C}{\eps^{21}} \right).
\end{equation}

\end{Theorem}

\begin{proof}
Let $f^{a>b},f^{b>c},f^{c>a} : \{-1,1\}^n \to \{-1,1\}$ be the three pairwise preference functions.
Let $\eta = \delta$ (where the values of $C$ will be determined later).
We will consider three cases:
\begin{itemize}
\item
There exist two voters $i \neq j \in [n]$ and two functions $f \neq g \in \{f^{a>b},f^{b>c},f^{c>a}\}$ such that
\begin{equation} \label{eq:case1}
I_i(f) > \eta, \quad I_j(g) > \eta.
\end{equation}
\item
For every two functions $f \neq g \in \{f^{a>b},f^{b>c},f^{c>a}\}$ and every $i \in [n]$, it holds that
\begin{equation} \label{eq:case2}
\min(I_i(f),I_i(g)) < \eta.
\end{equation}
\item
There exists a voter $j'$ such that for all $j \neq j'$
\begin{equation} \label{eq:case3}
\max(I_j(f^{a>b}),I_j(f^{b>c}),I_j(f^{c>a})) < \eta.
\end{equation}
\end{itemize}

First note that each $F$ satisfies at least one of the three conditions (\ref{eq:case1}), (\ref{eq:case2})
or (\ref{eq:case3}). Thus it suffices to prove the theorem for each of the three cases.

In~(\ref{eq:case1}), we have by Theorem~\ref{thm:two_inf} have that
\[
P(F) > \frac{1}{36} \eta^3.
\]
We thus obtain that $P(F) > \delta$ where $\delta$ is given in~(\ref{eq:del33}) 
by taking larger values $C'$ for $C$.

In case~(\ref{eq:case2}), by Theorem~\ref{thm:arrow_low_cross_unif} it follows that either there exist a function $G$ which
always put a candidate at top / bottom and $D(F,G) < \eps$ (if~(\ref{eq:non_top_bot_cross_unif}) holds), or
$P(F) > C \eps^{20} >> \delta$. 

Similarly in the remaining case~(\ref{eq:case3}), we have by Theorem~\ref{thm:arrow_one_inf} that either $D(F,G) < \eps$ or
$P(F) > C \eps^{20} >> \delta$. 
The proof follows.

\end{proof}

\section{Proof Concluded} \label{sec:k}
We now conclude the proof.

\begin{Theorem} \label{thm:k}
Consider voting on $k$ alternatives where voters vote uniformly at random from $S_k^n$.
Let $\frac{1}{100} > \eps > 0$. Then there exists a $\delta = \delta(\eps)$, such that for every $n$, if $F$ is a
constitution on $n$ voters satisfying:
\begin{itemize}
\item
IIA and
\item
$P(F) < \delta$,
\end{itemize}
then there exists $G \in \F_k(n)$ satisfying $D(F,G) < k^2 \eps$.

Moreover, one can take
\begin{equation} \label{eq:del3}
\delta = \exp \left(-\frac{C}{\eps^{21}} \right).
\end{equation}

\end{Theorem}

\begin{proof}
The proof follows by applying Theorem~\ref{thm:3} to triplets of alternatives.
Assume $P(F) < \delta(\eps)$.

Note that if $g_1,g_2 : \{-1,1\}^n \to \{-1,1\}$ are two different function each of which is either
a dictator or a constant function than $D(g_1,g_2) \geq 1/2$. Therefore for all $a,b$ it holds that
$D(f^{a>b},g) < \eps/10$ for at most one function $g$ which is either a dictator or a constant function.
In case there exists such function we let $g^{a>b} = g$, otherwise, we let $g^{a>b} = f^{a>b}$.

Let $G$ be the social choice function defined by the functions $g^{a>b}$. Clearly:
\[
D(F,G) < {k \choose 2} \eps < k^2 \eps.
\]
The proof would follow if we could show $P(G) = 0$ and therefore $G \in \F_k(n)$.

To prove that $G \in F_k(n)$ is suffices to show that for every set $A$ of three alternatives, it holds that
$G_A \in \F_3(n)$. Since $P(F) < \delta$ implies $P(F_A) < \delta$, Theorem~\ref{thm:3} implies that there exists a function $H_A \in F_3(n)$ s.t. $D(H_A,F_A) < \eps$.
There are two cases to consider:
\begin{itemize}
\item
$H_A$ is a dictator. This implies that $f^{a>b}$ is $\eps$ close to a dictator for each $a,b$ and therefore $f^{a>b} = g^{a>b}$ for all pairs $a,b$, so $G_A = H_A \in \F_3(n)$.
\item
There exists an alternative (say $a$) that $H_A$ always ranks at the top/bottom.
In this case we have that $f^{a>b}$ and $f^{c>a}$ are at most $\eps$ far from the constant functions $1$ and $-1$ (or $-1$ and $1$).
The functions $g^{a>b}$ and $g^{c>a}$ have to take the same constant values and therefore again we have that $G_A \in \F_3(n)$.
\end{itemize}

The proof follows.

\end{proof}

\begin{Remark}
Note that this proof is generic in the sense that it takes the quantitative Arrow's result for $3$ alternatives as a black box and
produces a quantitative Arrow result for any $k \geq 3$ alternatives.
\end{Remark}

\section{The class $\F_k(n)$} \label{sec:general_arrow}
In this section we prove Theorem~\ref{thm:general_arrow}. As noted before Wilson~\cite{Wilson:72} gave a partial characterization of functions satisfying IIA. Using a version of Barbera's lemma and the fact we consider only strict orderings we are able to give a complete characterization of the class $\F_k(n)$. For the discussion below it would be useful to say that 
the constitution $F$ is a {\em Degenerate} if there exists an alternative $a$ such that for all profiles $F$ ranks at the top (bottom).
The constitution $F$ is {\em Non Degenerate (ND)} if it is not degenerate.

\subsection{Different Pivots for Different Choices imply Non-Transitivity}
We begin by considering the case of $3$ candidates named $a,b,c$ and $n$ voters named $1,\ldots,n$. We first state Barbera's lemma in this case. 

\begin{Theorem} \label{thm:tech}
Consider a social choice function on $3$ candidates $a,b$ and $c$ and $n$ voters denoted $1,2,\ldots,n$. Assume that the social choice function satisfies that IIA condition and that there exists voters $i \neq j$ such that voter $i$ is pivotal for $f^{a>b}$ and voter $j$ is pivotal for $f^{b>c}$.
Then there exists a profile for which
$(f^{a>b}(x^{a>b}),f^{b>c}(x^{b>c}),f^{c>a}(x^{c>a}))$ is non-transitive.
\end{Theorem}

\begin{proof}
Without loss of generality assume that voter $1$ is pivotal for $f^{a>b}$ and voter $2$ is pivotal for $f^{b>c}$.
Therefore there exist $x_2,\ldots,x_n$ satisfying
\begin{equation} \label{eq:xpiv}
f^{a>b}(+1,x_2,\ldots,x_n) \neq f^{a>b}(-1,x_2,\ldots,x_n)
\end{equation}
and $y_1,y_3,\ldots,y_n$ satisfying
\begin{equation} \label{eq:ypiv}
f^{b>c}(y_1,+1,y_3,\ldots,y_n) \neq f^{b>c}(y_1,-1,y_3,\ldots,y_n).
\end{equation}
Let $z_1 = -y_1$ and $z_i = -x_i$ for $i \geq 2$.
By~(\ref{eq:xpiv}) and~(\ref{eq:ypiv}) we may choose $x_1$ and $y_2$ so that
\[
f^{a>b}(x) = f^{b>c}(y) = f(z),
\]
where $x = (x_1,\ldots,x_n),y=(y_1,\ldots,y_n)$ and $z=(z_1,\ldots,z_n)$.
Note further, that by construction for all $i$ it holds that
\[
(x_i,y_i,z_i) \notin \{(1,1,1),(-1,-1,-1)\},
\]
and therefore there exists a profile $\sigma$ such that
\[
x = x(\sigma), \quad y = y(\sigma), \quad z = z(\sigma).
\]
The proof follows.
\end{proof}





\subsection{$n$ voters, 3 Candidates}

In order to prove Theorem~\ref{thm:general_arrow} we need the following proposition regarding constitutions of a single voter.
\begin{Proposition} \label{prop:1}
Consider a constitution $F$ of a single voter and three alternatives $\{a,b,c\}$ which satisfies IIA and transitivity.
Then exactly one of the following conditions hold:
\begin{itemize}
\item
$F$ is constant. In other words, $F(\sigma) = \tau$ for all $\sigma$ and some fixed $\tau \in S(3)$.
\item
There exists an alternative $c$ such that $c$ is always ranked at the top (bottom) of the ranking and
$f^{a>b}(x) = x$ or $f^{a>b}(x) = -x$.
\item
$F(\sigma) = \sigma$ for all $\sigma$
\item
$F(\sigma) = -\sigma$ for all $\sigma$.
\end{itemize}
\end{Proposition}

\begin{proof}
Assume $F$ is not constant, then there exist two alternatives $a,b$ such that $f^{a>b}$ is not constant and therefore $f^{a>b}(x) = x$ or $f^{a>b}(x)=-x$.
Let $c$ be the remaining alternative. If $c$ is always ranked at the bottom or the top the claim follows.
Otherwise one of the functions $f^{a>c}$ or $f^{b>c}$ is not constant. We claim that in this case all three functions are non-constant. Suppose by way of contradiction that
$f^{c>a}$ is the constant $1$. This means that $c$ is always ranked on top of $a$. However, since $f^{a>b}$ is non-constant there
exists a value $x$ such that $f^{a>b}(x) = 1$ and similarly there exist a value $y$ such that $f^{b>c}(y) = 1$.
Let $\sigma$ be a ranking whose $a>b$ preference is given by $x$ and whose $b>c$ preferences are given by $y$.
Then $G(\sigma)$ satisfies that $a$ is preferred to $b$ and $b$ is preferred to $c$. Thus by transitivity it follows that
$a$ is preferred to $c$ - a contradiction. The same argument applied if $f^{c>a}$ is the constant $-1$ or if $f^{b>c}$ is a constant function.

We have thus established that all three functions $f^{a>b},f^{b>c}$ and $f^{c>a}$ are of the form $f(x)=x$ of $f(x)=-x$.
To conclude we want to show that all three functions are identical. Suppose otherwise. Then two of the functions have the same sign while the third has a different sign. Without loss of generality assume $f^{a>b}(x) = f^{b>c}(x) = x$ and $f^{c>a}(x)=-x$. Then looking at the profile $a>b>c$ we see that $\sigma' = F(\sigma)$ must satisfy $a>b$ and $b>c$ but also $c>a$ a contradiction. A similar proof applies when $f^{a>b}(x) = f^{b>c}(x) = -x$ and $f^{c>a}(x)=x$.
\end{proof}

\begin{Theorem} \label{thm:3c}
Any constitution on three alternatives which satisfies Transitivity, IIA and ND is a dictator.
\end{Theorem}

\begin{proof}
There are two cases to consider. The first case is where two of the functions $f^{a>b},f^{b>c}$ and $f^{c>a}$ are constant.
Without loss of generality assume that $f^{a>b}$ and $f^{b>c}$ are constant. Note that if $f^{a>b}$ is the constant $1$ and
$f^{b>c}$ is the constant $-1$ then $b$ is ranked at the bottom for all social outcomes in contradiction to the ND condition.
A similar contradiction is derived if $f^{a>b}$ is the constant $-1$ and $f^{b>c}$ is the constant $1$.
We thus conclude that $f^{a>b} = f^{b>c}$. However by transitivity this implies that $f^{c>a}$ is also a constant function and
$f^{c>a} = -f^{a>b}$.

The second case to consider is where at least two of the functions $f^{a>b},f^{b>c}$ and $f^{c>a}$ are not constant.
Assume without loss of generality that $f^{a>b}, f^{b>c}$ are non-constant. Therefore, each has at least one pivotal voter.
From Theorem~\ref{thm:tech} it follows that there exists a single voter $i$ such that each of the functions is either constant, or has a single pivotal voter $i$.
We thus conclude that $F$ is of the form $F(\sigma) = G(\sigma(i))$ for some function $G$. Applying Proposition~\ref{prop:1} shows that
either $G(\sigma)=\sigma$ or $G(\sigma)=-\sigma$ and concludes the proof.
\end{proof}

\subsection{The Characterization Theorem}
We now prove Theorem~\ref{thm:general_arrow}. Given a set of alternatives $A' \subset A$ and an alternative
$b \notin A$, we write $b \sim A'$ if there exist two alternatives $a,a'\in A_s$ and two profiles $\sigma$ and $\sigma'$ s.t. $F(\sigma)$ ranks $b$ above $a$ and $F(\sigma')$ ranks $a'$ above $b$. Note that if it does not hold that $b \sim A'$ then either $\{b\} >_F A'$ or $A' >_F \{b\}$.

We will use the following lemmas.
\begin{Lemma} \label{lem:gen_ind}
Let $F$ be a transitive constitution satisfying IIA and $A_1,\ldots,A_r,\{b\}$ disjoint sets of alternatives satisfying $A_1 >_F A_2 >_F \ldots >_F A_r$. Then either
\begin{itemize}
\item
There exists an $1 \leq s \leq r+1$ such
\begin{equation} \label{eq:simple_ind}
A_1 >_F \ldots >_F A_{s-1} > \{ b \} >_F A_s >_F \ldots >_F A_r,
\end{equation}
or
\item
There exist an $1 \leq s \leq r$ such that $b \sim A_r$ and
\begin{equation} \label{eq:complicated_ind}
A_1 >_F \ldots >_F A_{s} \cup \{ b \} >_F A_{s+1} >_F \ldots >_F A_r.
\end{equation}
\end{itemize}
\end{Lemma}

\begin{proof}
Consider first the case where for all $s$ it does not hold that $b \sim A_s$.
In this case for all $s$ either $b >_F A_s$ or $A_s >_F b$. Since $b >_F A_s$ implies
$b >_F A_{s+1} >_F \ldots$ and $A_{s'} >_F b$ implies $\ldots >_F A_{s'-1} >_F A_{s'} >_F b$ for all
$s,s'$ by transitivity, equation~(\ref{eq:simple_ind}) follows.

Next assume $b \sim A_s$. We argue that in this case
\[
\ldots >_F A_{s-1} >_F \{b\} >_F A_{s+1} >_F \ldots,
\]
which implies~(\ref{eq:complicated_ind}).

Suppose by contradiction that $b >_F A_{s+1}$ does not hold. Then there exists an element $a \in A_{s+1}$
and a profile $\sigma$ where $F(\sigma)$ ranks $a$ above $b$. From the fact that $b \sim A_s$ it follows
that there exist $c \in A_s$ and a profile $\sigma'$ where $F(\sigma')$ ranks $b$ above $c$ above $a$.
We now look at the constitution $F$ restricted to $B = \{a,b,c\}$. For each of $a,b,c$ there exist at least one profile where they are not at the top/bottom of the social outcome. It therefore follows that Theorem~\ref{thm:3c} applies
to $F_B$ and that $F_B$ is a dictator. However, the assumption that $A_s >_F A_{s+1}$ implies that $c >_F a$.
A contradiction. The proof that $F_{s-1} >_F b$ is identical.
\end{proof}

\begin{Lemma} \label{lem:gen3}
Let $F$ be a constitution satisfying transitivity and IIA. Let $A$ be a set of alternatives such that
$F_A$ is a dictator and $b \sim A$. Then $F_{A \cup \{b\}}$ is a dictator.
\end{Lemma}

\begin{proof}
Assume without loss of generality that $F_A(\sigma) = \sigma(i)$. Let $a \in A$ be such that there exist a profile
where $F$ ranks $a$ above $b$ and $c \in A$ be such there exists a profile where $a$ is ranked below $c$. Let
$B = \{a,b,c\}$. Then $F_B$ satisfies the condition of Theorem~\ref{thm:3c} and is therefore dictator.
Moreover since the $f^{a>c}(x)=x(i)$ it follows that $f^{a>b}(x)=x(i)$ and $f^{b>c}(x)=x(i)$.
Let $d$ be any other alternative in $A$. Let $B = \{a,b,d\}$. Then since $f^{a>b}(x) = f^{a>d}(x) = x(i)$, the conditions of Theorem~\ref{thm:3c} hold for $F_B$ and therefore $f^{b>d}(x) = x(i)$. We have thus concluded that
$F_{A \cup\{b\}}(\sigma) = \sigma$ for all $\sigma$ as needed. The proof for the case where $F_A(\sigma) = -\sigma$
is identical.
\end{proof}

Theorem~\ref{thm:3c} also immediately implies the following:
\begin{Lemma} \label{lem:gen2}
Let $F$ be a constitution satisfying transitivity and IIA. Let $A$ be a set of two alternatives such that
$F_A$ is not constant and $b \sim A$. Then $F_{A \cup \{b\}}$ is a dictator.
\end{Lemma}

We can now prove Theorem~\ref{thm:general_arrow}.
\begin{proof}
The proof is by induction on the number of alternatives $k$. The case $k=2$ is trivial. Either
$F$ always ranks $a$ above $b$ in which case $\{a\} >_F \{b\}$ as needed or $F$ is a non-constant function
in which case the set $A =\{a,b\}$ satisfies the desired conclusion.

For the induction step assume the theorem holds for $k$ alternatives and let $F$ be a constitution on $k+1$ alternatives which satisfies IIA and Transitivity. Let $B$ be a subset of $k$ of the alternatives and $b = A \setminus B$.

By by the induction hypothesis applied to $F_B$, we may write $B$ as a disjoint union of $A_1,\ldots,A_r$
such that $A_1 >_F A_2 > \ldots >_F A_r$ and such that if $A_s$ is of size $3$ or more then $F_{A_s}$ is a dictator and if $F_{A_s}$ is of size two then $F_{A_s}$ is non constant. We now apply Lemma~\ref{lem:gen_ind}. If~(\ref{eq:simple_ind}) holds then the proof follows. If~(\ref{eq:complicated_ind}) holds then the proof would follow once we show that $F_C$ is of the desired form where $C = A_s \cup \{b\}$. If $A_s$ is of size $1$ then from the definition of $\sim$ it follows that $F_{A_s \cup \{b\}}$ is non-constant as needed. If $A_s$ is of size $2$ then Lemma~(\ref{lem:gen2}) implies that $F_{A_s \cup \{b\}}$ is a dictator as needed and for the case of $A_s$ of size $3$ or more this follows from Lemma~(\ref{lem:gen3}). The proof follows.
\end{proof}

\section{Symmetric Distributions} \label{sec:sym}
In this section we provide some details on how to prove the results stated for general symmetric distributions. Most of the generalizations are straightforward. The main exception is Arrow theorem for low influences functions and the corresponding Gaussian result. These results require extension of the Invariance machinery and are developed in subsections~\ref{subsec:gauss_sym} and~\ref{subsec:low}.

\subsection{The Correlation Between $x^{a>b}$ and $x^{b>c}$}
The same proof of Lemma~\ref{lem:cor_vote} gives the following:

\begin{Lemma} \label{lem:cor_vote_general}
Assume that voters vote independently at random following a symmetric distribution $\mu$ $S(3)$ with minimal atom probability $\alpha$.
Then:
\begin{enumerate}
\item
For all $a,b$ and $i$ it holds that $\E[x^{a>b}(i)] = 0$.
\item
For all $i \neq j$ and all $a,b,c,d$ the variables $x^{a>b}(i)$ and $x^{c>d}(j)$ are independent.
\item
If $a,b,c$ are distinct then $|\E[x^{a>b}(i) x^{b>c}(i)]| \leq 1-4 \alpha$.
\end{enumerate}
\end{Lemma}

Similarly ro Lemma~\ref{lem:cor_vote2} we obtain
\begin{Lemma} \label{lem:cor_vote2gen}
Assume that voters vote from a symmetric distribution $\mu$ on $S(3)$.  Let $f = x^{c>a}$ and
Let $(Tf)(x^{a>b},x^{b>c}) = \E[f | x^{a>b},x^{b>c}]$. Then
\[
| Tf |_2 \leq \sqrt{1 - 4 \alpha}.
\]
\end{Lemma}

\begin{proof}
The proof is identical to the previous proof.
\end{proof}

\subsection{Two Influential Voters}
We briefly note that repeating the proofs of Lemma~\ref{lem:two_inf} and Theorem~\ref{thm:two_inf} we obtain the same results with
\begin{itemize}
\item
In Lemma~\ref{lem:two_inf} we obtain the lower bound
\[
\P[B] \geq \eps^{\frac{1}{2 \alpha}}.
\]
\item
In Theorem~\ref{thm:two_inf} we obtain the lower bound $P(F) > \beta^2 \eps^{\frac{1}{2 \alpha}}$, where $\beta = \alpha k! / 6$.
\end{itemize}

\subsection{Almost Transitive Functions}
We note that the same proof of Theorem~\ref{thm:3small} gives the following result for symmetric distributions.
\begin{Theorem} \label{thm:3smallsym}
Consider voting on $3$ alternatives where
each voter follows a symmetric voting distribution with minimal probability $\alpha$.
Let
\begin{equation} \label{eq:eps_bd_3small_sym}
\frac{\alpha^2}{9} > \eps > 0.
\end{equation}
For every $n$, if $F$ is a
constitution on $n$ voters satisfying:
\begin{itemize}
\item
IIA and
\item
\begin{equation} \label{eq:unif_small3n}
P(F) < \alpha^2 \eps^3 n^{-\frac{1}{2 \alpha}}.
\end{equation}
\end{itemize}
then there exists $G \in \F_3(n)$ satisfying $D(F,G) \leq 10 \eps$.
\end{Theorem}
Theorem~\ref{thm:3smallsym} implies in turn the second assertion of Theorem~\ref{thm:ksmall}.

\subsection{The Gaussian Arrow Theorem} \label{subsec:gauss_sym}
We start by proving a version of Theorem~\ref{thm:arrow_gauss} for symmetric distributions.

Since averaging maintains the expected value and covariances, we define:
\begin{eqnarray} \label{eq:gauss_sym}
\E[N_1^2] = \E[N_2^2] = \E[N_3^2] = 1,\\ \nonumber
\E[N_1 N_2] = \E[x^{a>b}(1) x^{b>c}(1)] := \rho_{1,2},\\ \nonumber
 \E[N_2 N_3] = \E[x^{b>c}(1) x^{c>a}(1)] := \rho_{2,3},\\ \nonumber
\E[N_3 N_1] = \E[x^{c>a}(1) x^{a>b}(1)] := \rho_{3,1}. \nonumber
\end{eqnarray}
We let $N(1),\ldots,N(n)$ be independent copies of $N$.
We write $\CalN = (N(1),\ldots,N(n))$ and for $1 \leq i \leq 3$ we write $\CalN_i = (N(1)_i,\ldots,N(n)_i)$. The variant of Theorem~\ref{thm:arrow_gauss} we prove is the following.

\begin{Theorem} \label{thm:arrow_gauss_sym}
For every $\eps > 0$ there exists a $\delta = \delta(\eps) > 0$ such that the following hold.
Let $f_1,f_2,f_3 : \R^n \to \{-1,1\}$.
Assume that for all $1 \leq i \leq 3$ and all $u \in \{-1,1\}$ it holds that
\begin{equation} \label{eq:non_dict_normal_sym}
\P[f_i(\CalN_i) = u, f_{i+1}(\CalN_{i+1}) =-u] \leq 1-\eps
\end{equation}
Then with the setup given in~(\ref{eq:gauss}) it holds that:
\[
\P[f_1(\CalN_1) = f_2(\CalN_2) = f_3(\CalN_3)] \geq \delta.
\]
Moreover, one may take $\delta = (\eps/2)^{1/2 \alpha^2}$.
\end{Theorem}

\begin{proof}
The proof is similar to the proof of Theorem~\ref{thm:arrow_gauss}. Note that if $\P[f_2 = 1] > \eps/2$ and $\P[f_3 = 1] > \eps/2$ then:
\[
\P[f_2 = 1, f_3 = 1] > (\eps/2)^{1/2\alpha}.
\]
Again we define $\CalM_1, \CalM_2$ so that $\CalM_2$ is uncorrelated with $\CalN$. Using Lemma~\ref{lem:cor_vote2gen} we obtain
that the correlation between $\CalM_1(i)$ and $\CalN(i)$ is at most $\sqrt{1-4\alpha}$. Using
$1-\sqrt{1-4\alpha} \geq 2 \alpha$ one then obtains
\[
\P[f_1 = 1, f_2 = 1, f_3 = 1] > (\eps/2)^{\frac{1}{2\alpha^2}}.
\]
\end{proof}

\subsubsection{Kalai's Formula and $[-1,1]$ Valued Votes.}
In this subsection we will give a more detailed description of the functions that achieve minimum probability of a paradox
in the Gaussian case. We will first generalize Lemma~\ref{lem:kalai}.



The same proof of Lemma~\ref{lem:kalai} gives the following:
\begin{Lemma} \label{lem:kalai_gauss}
Consider the setup of Theorem~\ref{thm:arrow_gauss}. Then:
\begin{equation} \label{eq:PG}
\P[f_1(\CalN_1) = f_2(\CalN_2) = f_3(\CalN_3)] =
\frac{1}{4} \left( 1 + \E[f_1(\CalN_1)f_2(\CalN_2)] + \E[f_2(\CalN_2) f_3(\CalN_3)] +
\E[f_3(\CalN_3)f_1(\CalN_1)] \right)
\end{equation}
\end{Lemma}



Given the basic voting setup, we {\em define} $P(f_1,f_2,f_3)$ for three function $f_1,f_2,f_3 : \{-1,1\}^n \to [-1,1]$ by letting
\begin{eqnarray*}
P(f_1,f_2,f_3) &=& \E[s(f_1(x^{a>b},f_2(x^{b>c}),f_3(x^{c>a})] \\ &=&
\frac{1}{4} \left( 1 + \E[f_1(x^{a>b})f_2(x^{b>c})] + \E[f_2(x^{b>c})f_3(x^{c>a})] +
\E[f_3(x^{c>a})f_1(x^{a>b})] \right).
\end{eqnarray*}
Similarly, given three function $f_1,f_2,f_3 : \R^n \to [-1,1]$ we {\em define}
\begin{eqnarray*}
P(f_1,f_2,f_3) &=& \E[s(f_1(\CalN_1,\CalN_2,\CalN_3))] \\
               &=& \frac{1}{4} \left( 1 + \E[f_1(\CalN_1)f_2(\CalN_2))] + \E[f_2(\CalN_2)f_3(\CalN_3)] +
\E[f_3(\CalN_3)f_1(\CalN_1)] \right),
\end{eqnarray*}
where $\CalN_i$ are define in~(\ref{eq:gauss}).

We will use the following lemma.
\begin{Lemma} \label{lem:s}
The functions $s : [0,1]^3 \to \R$ defined by $s(x,y,z) = 1/4(1 + xy + yz + zx)$ takes values in $[0,1]$.
Moreover, if $x,y,z$ take value at most $(1-\eps)$ (at least $-1+\eps$) then the function takes the value at least $\eps^2/4$.
\end{Lemma}

\begin{proof}
Both statements follow from the fact that $s$ takes the values $0,1$ on the vertices of the cube and that $s$ is affine in each of the coordinates. Note that on the vertices of the cube $[-1,1-\eps]^3$ the function $s$ takes the values $1,\eps/2,\eps^2/4$ and
$1/4 + 3/4 (1-\eps)^2$.
\end{proof}

We can now state and prove the $[-1,1]$ version of the Gaussian Arrow theorem.
\begin{Theorem} \label{thm:arrow_gauss_gen}
For every $\eps > 0$ there exists a $\delta > 0$ such that the following hold.
Let $f_1,f_2,f_3 : \R^n \to [-1,1]$.
Assume that for all $1 \leq i \leq 3$ and all $u \in \{-1,1\}$ it holds that
\begin{equation} \label{eq:non_dict_normal_gen}
\min(u \E[f_i], -u \E[f_{i+1}]) \leq 1-2\eps
\end{equation}
Then with the setup given in~(\ref{eq:gauss}) it holds that:
\[
\P(f_1,f_2,f_3) \geq \delta.
\]
Moreover, in the uniform case one may take $\delta = (\eps/2)^{20}$.
In the general case one may take $\delta = (\eps/2)^{2+1/(2 \alpha^2)}$.
\end{Theorem}

\begin{proof}
We will prove the claim for the uniform case.
Again we consider two cases.
The first case to consider is where there exist a function $f_i$ and $u \in \{-1,1\}$ with $\P[|f_i-u| < \eps] > 1-\eps/2$.
Without loss of generality assume that $\P[f_1 \geq 1-\eps] > 1-\eps/2$ and note that this implies that $\E[f_1] > 1-2\eps$.
By~(\ref{eq:non_dict_normal_gen}) it therefore follows that $\E[f_2] > -1+2\eps$ and $\E[f_3] > -1+2\eps$.
We thus conclude that $\P[f_2 > -1+\eps] > \eps/2$ and $\P[f_3 > -1+\eps] > \eps/2$.
Let $B_1,B_2,B_3$ denote the sets where $f_1,f_2,f_3$ take value greater than $-1+\eps$ and let $B$ denote the intersection of these sets. Repeating the argument in the previous lemma we obtain
\[
\P[B] \geq (\eps/2)^{18}.
\]
By Lemma~\ref{lem:s} on the event $B$, the value of $s(f_1,f_2,f_3)$, is at least $\eps^2/2$, and on the complement of $B$, it is non-negative.
We thus conclude
\[
P(f_1,f_2,f_3) \geq (\eps/2)^{20}.
\]
In the second case, all functions satisfy $\P[f_i \leq 1-\eps] \geq \eps/2$. Then, letting $A_i$ denote the event where
$f_i \leq 1-\eps$ and repeating the argument above, we obtain the same bound.
The proof for non-uniform distributions is identical.
\end{proof}

\begin{Remark} \label{rem:gen_gauss}
We briefly note that Theorem~\ref{thm:arrow_gauss_gen} and the other theorems proven in this section hold in further generality, where voters vote independently according to any (perhaps not symmetric) distribution over the rankings where the probability of any ranking
is at least $\alpha$. The proof of these extensions is identical to the proofs provided here.
\end{Remark}

\subsection{Arrow Theorem for Low Influence Functions} \label{subsec:low}
The proof for symmetric distributions is way more involved than the proof for uniform distributions. The main difference between the two cases is while in the uniform case in the expansion~(\ref{eq:PF}) each correlation factor is asymptotically minimized by symmetric monotone threshold functions and therefore the overall expression is also minimized by symmetric monotone threshold functions.

In the general symmetric case, it is impossible to apply invariance to pairs of functions as one of the correlations parameters
in~(\ref{eq:PF}) may be positive in which case it is {\em maximized} (rather than minimized) by monotone symmetric threshold functions. To deal with this case we therefore derive an appropriate extension of invariance which may be of independent interest.
Roughly speaking the extension establishes a map
$\Psi$ mapping $B = \{ f : \{-1,1\}^n \to [-1,1] \}$ to $G = \{ f : \R^n \to [-1,1] \}$ such that for functions with low
influences $\E[fg]$ is close to $\E[\Psi(f) \Psi(g)]$ where the first expected value is with respect to two correlated input from the  uniform measure on $\{-1,1\}$ and the second is with respect to two correlated Gaussian in $\R^n$.


\subsubsection{Symmetric Distributions}

For the statement below we recall the notion of low degree influence. For a functions $f : \{-1,1\}^n \to \R$ where
$\{-1,1\}^n$ is equipped with the uniform measure, the degree $d$ influence of the $i$'th variable of $f$ is defined by:
\begin{equation} \label{eq:low_inf}
I^{\leq d}(f) = \sum_{S: |S| \leq d, i \in S} \hat{f}^2(S).
\end{equation}
Obviously $I_i^{\leq d}(f) \leq I(f)$. The usefulness of $I_i^{\leq d}$ to be used in the next subsection comes from the fact that
\begin{equation} \label{eq:low_inf_sum}
\sum_i I_i^{\leq d}(f) \leq d \cdot \Var[f].
\end{equation}

\subsubsection{Invariance Result}
Our starting point will be the following extensions of results from~\cite{MoOdOl:09}
and~\cite{Mossel:09}.

\begin{Theorem} \label{thm:invariance}
For all $\eps, -1 < \rho < 1$ the following holds.
Consider the space $\{-1,1\}^n$ equipped with the uniform measure and the space $\R^n$ equipped with the Gaussian measure.
Then for every function $f : \{-1,1\}^n \to [-1,1]$ there exists a function $\tilde{f} : \R^n \to [-1,1]$ such that the following hold.
Consider $(X,Y)$ distributed in $\{-1,1\}^n \times \{-1,1\}^n$ where $(X_i,Y_i)$ are independent with
\[
\E[X_i] = \E[Y_i] = 0, \quad \E[X_i^2] = \E[Y_i^2] = 1, \quad \E[X_i Y_i] = \rho.
\]
Consider $(N,M)$ jointly Gaussian and distributed in $R^n \times R^n$ with $(N_i,M_i)$ independent with
\[
\E[N_i] = \E[M_i] = 0, \quad \E[N_i^2] = \E[M_i^2] = 1, \quad \E[N_i M_i] = \rho.
\]
Then
\begin{itemize}
\item
For the constant functions $1$ and $-1$ it holds that $\tilde{1} = 1$ and $\tilde{-1} = -1$.
\item
If $f$ and $g$ are two functions such that for all $i$, it holds that $\max(I_i^{\log(1/\tau)}(f),I_i^{\log(1/\tau)}(g)) < \tau$ then
\begin{equation} \label{eq:inv}
|\E[f(X) g(Y)] - \E[\tilde{f}(N) \tilde{g}(M)]| \leq \eps.
\end{equation}
if
\begin{equation} \label{eq:tau_bound}
\tau \leq  \tau(\eps,|\rho|) := \eps^{C \frac{\log(1/\eps)}{(1-|\rho|)\eps}},
\end{equation}
for some absolute constant $C$.
\end{itemize}
\end{Theorem}

\begin{proof}
We briefly explain how does this follow from the~\cite{MoOdOl:09} and~\cite{Mossel:09}.
Given $f$, we take small $\eta$ and look $T_{1-\eta} f$, where
$T$ is the Bonami-Beckner operator. For small $\eta$ and every two functions $f$ and $g$ it holds that
$\E[f(X) g(Y)]$ is $\eps/4$ close to $\E[T_{1-\eta}f(X) T_{1-\eta}g(Y)]$. By Lemma 6.1 in~\cite{Mossel:09} this can be done with
\[
\eta = C \frac{(1-|\rho|)\eps}{\log (1/\eps)}.
\]
$T_{1-\eta} f$ is given by a multi-linear polynomial which we can also write in terms of Gaussian random variables. Let's call the Gaussian polynomial $f'$. The polynomial $f'$ has the same expected value as $f$ but in general it takes values in all of $\R$.
Similarly for different $f$ and $g$ we have $\E[f' g'] = \E[T_{1-\eta}f(X) T_{1-\eta}g(Y)]$.
We let $\tilde{f}(x) = f'(x)$ if $|f'(x)| \leq 1$ and $f'(x) = 1$ ($f'(x) = -1$) if $f'(x) \geq 1$ ($f'(x) \leq -1$). It is easy to see that $\tilde{1} = 1$ and $\tilde{-1} = -1$.

By Theorem 3.20 in~\cite{MoOdOl:09} it follows that $\E[(f' - \tilde{f})^2] < \eps^2/16$  if all
influences of $f$ are bounded by $\tau$ given in~(\ref{eq:tau_bound}).
An immediate application of Cauchy-Schwartz implies that $\E[f'(N) g'(M)]$ is at most $\eps/2$ far from $\E[\tilde{f}(N) \tilde{f}(M)]$. We thus obtain Theorem~\ref{thm:invariance}.
\end{proof}

\subsubsection{Arrow Theorem for Low Influence Functions}
We can prove a quantitative Arrow theorem for low influence functions.
For the statement from this point on, it would be useful to denote
\[
\tau_{\alpha}(\delta) := \delta^{C \frac{\log(1/\delta)}{\alpha \delta}}, \quad
\tau(\delta) := \tau_{1/3}(\delta),
\]
for some absolute constant $C$.

\begin{Theorem} \label{thm:arrow_low_inf}
For every $\eps > 0$ there exists a $\delta(\eps) > 0$ and a $\tau(\delta) > 0$ such that the following hold.
Let $f_1,f_2,f_3 : \{-1,1\}^n \to [-1,1]$.
Assume that for all $1 \leq i \leq 3$ and all $u \in \{-1,1\}$ it holds that
\begin{equation} \label{eq:non_dict_low_inf}
\min(u \E[f_i], -u \E[f_{i+1}]) \leq 1-3\eps
\end{equation}
and for all $1 \leq i \leq 3$ and $1 \leq j \leq n$ it holds that
\[
I_j^{\log(1/\tau)}(f_i) < \tau,
\]
Then it holds that
\[
\P(f_1,f_2,f_3) > \delta.
\]
Moreover, assuming the uniform distribution, one may take:
\[
\delta = \frac{1}{4}(\eps/2)^{20}, \quad \tau = \tau(\delta,\frac{1}{3}).
\]
And assuming a general symmetric voting distribution with a minimal probability $\alpha$ for every
permutation, one can take:
\[
\delta = \frac{1}{4}(\eps/2)^{2+1/(2 \alpha^2)}, \quad \tau = \tau(\delta,1-4\alpha).
\]
\end{Theorem}

\begin{proof}
Let $g_1 = \tilde{f_1},g_2 = \tilde{f_2}$ and $g_3 = \tilde{f_3}$, the functions whose existence is guaranteed by Theorem~\ref{thm:invariance}. We will apply the theorem for the pairs of functions $(f_1,f_2),(f_2,f_3)$ and
$(f_3,f_1)$ and the correlations given in~(\ref{eq:gauss}).
Taking $\rho = 0$ and noting that $T_{0} 1 = 1$ and $\tilde{1} = 1$, we conclude that for all $i$ it holds that $|\E[f_i] - \E[g_i]| < \eps$. It therefore follows from~(\ref{eq:non_dict_low_inf}) that the
functions $g_i$ satisfy (\ref{eq:non_dict_normal_gen}) and therefore
$P(g_1,g_2,g_3) \geq 4 \delta$, where the
correlations between the $g_i$'s are given by~(\ref{eq:gauss}).

Recall that:
\[
P(g_1,g_2,g_3) = \frac{1}{4}\left( \E[g_1(\CalN_1) g_2(\CalN_2)] + \E[g_2(\CalN_2) g_3(\CalN_3)] +
                                   \E[g_3(\CalN_3) g_1(\CalN_1)] \right).
\]
Applying theorem~\ref{thm:invariance} we see that
\[
|\E[g_1(\CalN_1) g_2(\CalN_2)] - \E[f_1(x^{a>b}) f_2(x^{b>c})]| < \delta,
\]
and similarly for the other expectations. We therefore conclude that
\[
P(f_1,f_2,f_3) > P(g_1,g_2,g_3) - 3 \delta/4 > \delta,
\]
as needed.
\end{proof}

\subsubsection{Arrow Theorem For Low Cross Influences Functions}
Our final result in the low influence realm deals with the situation that for each coordinate, at most one function has large influence while the two others have small influences. Such a case occurs for example when one function is a function of a small number of voters while the two others are majority type functions.
The main result of the current subsection shows that indeed is such situation there is a good probability of a paradox. The proof is based on extending an averaging argument from~\cite{Mossel:09}.

\begin{Theorem} \label{thm:arrow_low_cross}
For every $\eps > 0$ there exists a $\delta(\eps) > 0$ and a $\tau(\delta) > 0$
such that the following hold.
Let $f_1,f_2,f_3 : \{-1,1\}^n \to [-1,1]$.
Assume that for all $1 \leq i \leq 3$ and all $u \in \{-1,1\}$ it holds that

\begin{equation} \label{eq:non_top_bot_cross}
\min(u \E[f_i], -u \E[f_{i+1}]) \leq 1-3\eps
\end{equation}
and for all $j$ it holds that
\begin{equation} \label{eq:cross_inf}
|\{ 1 \leq i \leq 3 : I_j^{\log^2{1/\tau}}(f_i) > \tau \}| \leq 1.
\end{equation}
Then it holds that
\[
\P(f_1,f_2,f_3) \geq \delta.
\]
Moreover, assuming the uniform distribution, one may take:
\[
\delta = \frac{1}{8}(\eps/2)^{20}, \quad \tau = \tau(\delta,1/3).
\]
Assuming a general symmetric voting distribution with a minimal distribution $\alpha$ for every
permutation, one can take:
\[
\delta = \frac{1}{8}(\eps/2)^{2+1/(2 \alpha^2)}, \quad \tau = \tau(\delta,1-4 \alpha).
\]
\end{Theorem}

The proof will use the following lemma which is a special case of a lemma from~\cite{Mossel:09}.
\begin{Lemma} \label{lem:average}
Let $\mu$ be a distribution on $\{-1,1\}^2$ with uniform marginals.
Let $f_1,f_2 : \{-1,1\}^n \to [0,1]$. Let $S \subset [n]$ be a set
of coordinates such that for each $i \in S$ at most one of the functions
$f_1,f_2$ have $I_i(f_j) > \eps$. Define
\[
g_i(x) = \E[f_i(Y) | Y_{[n] \setminus S} = x_{[n] \setminus S}].
\]
Then the functions $g_i$ do not depend on the coordinates in $S$, are $[0,1]$ valued, satisfy $\E[g_i] = \E[f_i]$ and
\[
|\E[f_1(X) f_2(Y)]] - \E[g_1(X) g_2(Y)]| \leq |S| \sqrt{\eps},
\]
where $(X_i,Y_i)$ are independent distributed according to $\mu$.
\end{Lemma}

\begin{proof}
Recall that averaging over a subset of the variables preserves expected value.
It also maintains the property of taking values in $[0,1]$
and decreases influences. Thus it suffices to prove the claim for the case
where $|S| = 1$. The general case then follows by induction.

So assume without loss of generality that $S=\{1\}$ consists of the first
coordinate only and that $I_2(f_2) \leq \eps$, so
that $\E[(f_2-g_2)^2] \leq \eps$. Then by Cauchy-Schwartz we have
$\E[|f_2-g_2|] \leq \sqrt{\eps}$ and
using the fact that the functions are bounded in $[0,1]$ we obtain
\begin{equation} \label{eq:eps_r_wise}
|\E[f_1 f_2 - f_1 g_2]| \leq \sqrt{\eps}.
\end{equation}
Let us write $\E_1$ for the expected value with respect to the first variable.
Recalling that the $g_i$ do not depend on the first variable we obtain that
\[
\E_1[f_1 g_2] = g_2 \E_1[f_1] = g_1 g_2.
\]
This implies that
\begin{equation} \label{eq:prod_r_wise}
\E[f_1 g_2] = \E[g_1 g_2],
\end{equation}
and the proof follows from~(\ref{eq:eps_r_wise}) and(\ref{eq:prod_r_wise}).
\end{proof}

We can now prove Theorem~\ref{thm:arrow_low_cross}.

\begin{proof}
The proof will use the fact that the sum of low-degree influences~(\ref{eq:low_inf_sum}) together with the fact that averaging makes (standard influences) smaller. In order to work with these two notions of influences simultaneously we begin by replacing each functions $f_i$ with the function
$T_{1-\eta} f_i$ where as in Theorem~\ref{thm:invariance} we let
\[
\eta = C_1 \frac{ \alpha \delta}{\log (1/\delta)},
\]
where $C_1$ is large enough so that
\[
|\E[f_1(x^{a>b}) f_2(x^{b>c})] - \E[T_{1-\eta} f_1(x^{a>b}) T_{1-\eta} f_2(x^{b>c})]| < \delta/10,
\]
and similarly for other pairs of functions.
Thus it suffices to prove the theorem assuming that all function have total Fourier weight at most
$(1-\eta)^{2r}$ above level $r$ and therefore all functions $f_i$ and all variables $j$ satisfy
$I_j(f_i) \leq I_j^{\leq r}(f_i) + (1-\eta)^{2r}$.

Let $C_2$ be chosen so that the statement of Theorem~\ref{thm:arrow_low_inf} holds for $\tau$, where
\[
\tau = \delta^{C_2 \frac{\log(1/\delta)}{ \alpha \delta}}.
\]
Let
\[
R = \log(1/\tau), \quad
R' = \log^2(1/\tau).
\]
and choose $C_2$ and $C_3$ large enough so that
\[
\tau' = \delta^{C_3 \frac{\log(1/\delta)}{\alpha \delta}},
\]
satisfies
\[
\frac{3 R}{\tau} \sqrt{\tau' + (1-\eta)^{2R'}} \leq \frac{\delta}{16}.
\]

Assume that $f_i$ satisfy~(\ref{eq:cross_inf}), i.e., for all $j$:
\[
|\{ 1 \leq i \leq 3 : I_j^{R'}(f_i) > \tau' \}| \leq 1.
\]

We will show that the statement of the theorem holds for $f_i$.
For this let
\[
S_i = \{ i : I_i^{\leq R}(f_i) > \tau \},
\]
and $S = S_1 \cup S_2 \cup S_3$.
Since $R' \geq R$ and $\tau' \leq \tau$, the sets $S_i$ are disjoint.

Moreover,each of the sets $S_i$ is of size at most $\frac{R}{\tau}$.

Also, if $j \in S$ and $I_j^{\leq R}(f_i) > \tau$ then for $i \neq i'$ it holds that
$I_j^{\leq R}(f_{i'}) < \tau'$
and therefore $I_j(f_{i'}) \leq \tau' + (1-\gamma)^{2 R'}$.
In other words, for all $j \in S$ we have that at least two of the functions $f_i$ satisfy
$I_j(f_i) \leq \tau$.

We now apply Lemma~\ref{lem:average} with
\[
\bar{f}_i(x) = \E[f_i(X) | X_{[n] \setminus S} = x_{[n] \setminus S}].
\]
We obtain that for any pair of functions $f_i,f_{i+1}$ it holds that
\begin{equation} \label{eq:trunc_double_coef}
|\E[f_{i} f_{i+1}] - \E[\bar{f}_i \bar{f}_{i+1}]| \leq
\frac{2 R}{\tau} \sqrt{\tau' + (1-\eta)^{2 R'}}  \leq \frac{\delta}{16}.
\end{equation}
Note that the functions $\bar{f}_i$ satisfy that
$\max_{i,j}I_j(\bar{f}_i)) \leq \tau$.
This implies that the results of Theorem~\ref{thm:arrow_low_inf}
hold for $\bar{f}_i$. This together with~(\ref{eq:trunc_double_coef}) implies
the desired result.
\end{proof}

\begin{Remark} \label{rem:gen_inf}
 We note that Theorem~\ref{thm:arrow_low_cross} and the other theorems proven in this section hold in further generality, where voters vote independently according to any (perhaps not symmetric) distribution over the rankings where the probability of any ranking
is at least $\alpha$ with bounds on $\tau$ that are somewhat worse than those obtained here. The proof of these extensions is similar to the proofs presented here (recall Remark~\ref{rem:gen_gauss}). The main difference is since now the distributions of $x^{a>b}$ etc. are biased, the applications of invariance principle results in somewhat worse results.
\end{Remark}

\subsection{One Influential Variable} 
We note that Theorem~\ref{thm:arrow_one_inf} holds as stated for symmetric distributions with $\alpha$ being the minimum probability over all permutations and 
\[
\delta = (\eps/2)^{2+1/(2 \alpha^2)}, \quad \tau = \tau(\delta,1-4\alpha).
\]
The only difference in the proof is that instead of Theorem~\ref{thm:arrow_low_cross_unif}.
we use Theorem~\ref{thm:arrow_low_inf}.

\subsection{Quantitative Arrow Theorem for $3$ Candidates} \label{subsec:3}
We briefly state the generalization of Theorem~\ref{thm:3} to symmetric distributions.

\begin{Theorem} \label{thm:3g}
The statement of the Theorem~\ref{thm:3} holds true for symmetric distributions on $3$ alternatives with minimum probability for each
ranking $\alpha$ with
\begin{equation} \label{eq:del43}
\delta = \exp \left(-\frac{C_1}{\alpha \eps^{C_2(\alpha)}} \right),
\end{equation}
where $C_2(\alpha) = 3+1/(2 \alpha^2)$.
\end{Theorem}

\begin{proof}
The proof is identical. 
In case~(\ref{eq:case1}), we now have 
\[
P(F) > \alpha^2 \eta^3.
\]
We thus obtain that $P(F) > \delta$ where $\delta$ is given~(\ref{eq:del43})
by taking larger values $C'_1$ and $C'_2$ for $C_1$ and $C_2$.

In case~(\ref{eq:case2}), by Theorem~\ref{thm:arrow_low_cross} it follows that either there exist a function $G$ which
always put a candidate at top / bottom and $D(F,G) < \eps$ (if~(\ref{eq:non_top_bot_cross}) holds), or
$P(F) > C \eps^{C_2(\alpha)} >> \delta$.

Similarly in the remaining case~(\ref{eq:case3}), we have by version of Theorem~\ref{thm:arrow_one_inf} for symmetric distributions 
that either $D(F,G) < \eps$ or $P(F) > C \eps^{C_2(\alpha)} >> \delta$. The proof follows.

\end{proof}

\subsection{Proof Concluded} 
The general version of Theorem~\ref{thm:k} reads:
\begin{Theorem} \label{thm:kg}
Theorem~\ref{thm:k} holds for symmetric distributions on $k$ alternatives with minimum probability for each
ranking $\beta$ and $\alpha=k!\beta/6$ and 
\begin{equation} \label{eq:del4}
\delta = \exp \left(-\frac{C_1}{\alpha \eps^{C_2(\alpha)}} \right),
\end{equation}
where $C_2(\alpha) = 3+1/(2 \alpha^2)$.
\end{Theorem}

\begin{proof}
The proof follows by applying Theorem~\ref{thm:3g} to triplets of alternatives as before.
Note that when restricting to $3$ alternatives, the minimum probability assigned to each order is at least $\alpha$. 
\end{proof}

\section{Open Problems}
As a conclusion we want to mention some natural open problems.

\begin{itemize}

\item
We believe that the results obtained here hold also for non-symmetric distributions of rankings as long as the probability of every ranking is bounded below by some constant $\alpha$. Recalling remarks~\ref{rem:gen_gauss} and~\ref{rem:gen_inf}, we see that the main challenge in extending the results to this setup is extending the proof for the case where two different functions $f$ and $g$ have two different influential voters. The problem in extending this result is the lack of inverse-hyper-contraction results for biased measures on $\{-1,1\}^n$. Deriving such estimates is of independent interest.

\item
A second natural problem is to attempt and obtain other quantitative results in social choice theory using Fourier methods.
A natural candidate is the Gibbard-Satterthwaite Theorem~\cite{Gibbard:73,Satterthwaite:75}. A first quantitative estimates for $3$ alternatives was obtained in~\cite{FrKaNi:08}. As mentioned before, the results of~\cite{FrKaNi:08} are limited in the sense that they require neutrality and apply only to 3 candidates. It is interesting to explore if the full quantitative version of Arrow theorem proven here will allow to obtain stronger quantitative version of the Gibbard-Satterthwaite Theorem.

\end{itemize}

\bibliographystyle{abbrv}
\bibliography{all,my}

\begin{thebibliography}{10}

\bibitem{Arrow:50}
K.~Arrow.
\newblock A difficulty in the theory of social welfare.
\newblock {\em J. of Political Economy}, 58:328--346, 1950.

\bibitem{Arrow:63}
K.~Arrow.
\newblock {\em Social choice and individual values}.
\newblock John Wiley and Sons, 1963.

\bibitem{Barbera:80}
S.~Barbera.
\newblock Pivotal voters: A new proof of arrow's theorem.
\newblock {\em Economics Letter}, 6:13--16, 1980.

\bibitem{Borell:82}
C.~Borell.
\newblock Positivity improving operators and hypercontractivity.
\newblock {\em Math. Zeitschrift}, 180(2):225--234, 1982.

\bibitem{Borell:85}
C.~Borell.
\newblock Geometric bounds on the {O}rnstein-{U}hlenbeck velocity process.
\newblock {\em Z. Wahrsch. Verw. Gebiete}, 70(1):1--13, 1985.

\bibitem{DiMoRe:06}
I.~Dinur, E.~Mossel, and O.~Regev.
\newblock Conditional hardness for approximate coloring.
\newblock In {\em Proceedings of the thirty-eighth annual ACM symposium on
  Theory of computing (STOC 2006)}, pages 344--353, 2006.

\bibitem{FrKaNi:08}
E.~Friedgut, G.~Kalai, and N.~Nisan.
\newblock Elections can be manipulated often.
\newblock In {\em Proceedings of the 49th Annual IEEE Symposium on Foundations
  of Computer Science (FOCS)}, pages 243--249, 2009.

\bibitem{Gibbard:73}
A.~Gibbard.
\newblock Manipulation of voting schemes: a general result.
\newblock {\em Econometrica}, 41(4):587–--601, 1973.

\bibitem{GoGoRo:96}
O.~Goldreich, S.~Goldwasser, and D.~Ron.
\newblock Property testing and its connection to learning and approximation.
\newblock {\em Journal of the Association for Computing Machinery},
  45(4):653--750, 1996.

\bibitem{Kalai:02}
G.~Kalai.
\newblock {A Fourier-theoretic perspective on the Concordet paradox and Arrow's
  theorem}.
\newblock {\em {Adv.\ in Appl.\ Math.}}, 29(3):412--426, 2002.

\bibitem{Keller:09}
N.~Keller.
\newblock On the probability of a rational outcome for generalized social
  welfare functions on three alternatives (submitted).
\newblock Submitted. Availible at www.ma.huji.ac.il/$\sim$nkeller, 2009.

\bibitem{Mossel:08}
E.~Mossel.
\newblock Gaussian bounds for noise correlation of functions and tight analysis
  of long codes.
\newblock In {\em Foundations of Computer Science, 2008 (FOCS 08)}, pages
  156--165. IEEE, 2008.

\bibitem{Mossel:09}
E.~Mossel.
\newblock Gaussian bounds for noise correlation of functions.
\newblock To appear in GAFA. Available at Arxiv math/0703683, 2009.

\bibitem{MoOdOl:05}
E.~Mossel, R.~O'Donnell, and K.~Oleszkiewicz.
\newblock Noise stability of functions with low influences: invariance and
  optimality (extended abstract).
\newblock In {\em 46th Annual IEEE Symposium on Foundations of Computer Science
  (FOCS 2005), 23-25 October 2005, Pittsburgh, PA, USA, Proceedings}, pages
  21--30. IEEE Computer Society, 2005.

\bibitem{MoOdOl:09}
E.~Mossel, R.~O'Donnell, and K.~Oleszkiewicz.
\newblock Noise stability of functions with low influences: invariance and
  optimality.
\newblock To appear in Ann. Math., 2009.

\bibitem{MORSS:06}
E.~Mossel, R.~O'Donnell, O.~Regev, J.~E. Steif, and B.~Sudakov.
\newblock Non-interactive correlation distillation, inhomogeneous {M}arkov
  chains, and the reverse {B}onami-{B}eckner inequality.
\newblock {\em Israel J. Math.}, 154:299--336, 2006.

\bibitem{Rotar:79}
V.~I. Rotar$'$.
\newblock Limit theorems for polylinear forms.
\newblock {\em J. Multivariate Anal.}, 9(4):511--530, 1979.

\bibitem{RubinfeldSudan:96}
R.~Rubinfeld and M.~Sudan.
\newblock Robust characterizations of polynomials with applications to program
  testing.
\newblock {\em SIAM J. Comput.}, 25(2):252--271, 1996.

\bibitem{Satterthwaite:75}
M.~A. Satterthwaite.
\newblock {Strategy-proofness and Arrow's Conditions: Existence and
  Correspondence Theorems for Voting Procedures and Social Welfare Functions}.
\newblock {\em J. of Economic Theory}, 10:187--–217, 1975.

\bibitem{Wilson:72}
R.~Wilson.
\newblock Social choice theory without the pareto principle.
\newblock {\em Journal of Economic Theory}, 5(3):478--486, 1972.

\end{thebibliography}
\end{document}